\documentclass[11pt]{amsart}
\usepackage{amscd}
\usepackage{amsfonts}
\usepackage{multicol}
\usepackage{color}
\usepackage{amsmath,amssymb,amsthm,latexsym}
\usepackage{amscd}

\usepackage{enumerate}

\newtheorem{theorem}{Theorem}[section]

\newtheorem{definition}[theorem]{Definition}
\newtheorem{corollary}[theorem]{Corollary}
\newtheorem{lemma}[theorem]{Lemma}
\usepackage{amsmath}
\usepackage{amsfonts}
\usepackage{amsmath,amsthm}
\usepackage{amscd}
\usepackage[all]{xy}
\numberwithin{equation}{section}
\theoremstyle{remark}
\newtheorem{remark}[theorem]{Remark}
\newtheorem{example}[theorem]{\bf Example}
\newcommand{\R}{\mathbb{R}}
\newcommand{\C}{\mathbb{C}}

\newcommand{\blue}{\textcolor{blue}}

\def\dd{\mathrm{d}}
\begin{document}

\title[Weierstrass-Kenmotsu representation  of Willmore surfaces]{\bf{
Weierstrass-Kenmotsu representation of Willmore surfaces in spheres}}
\author{Josef F. Dorfmeister, Peng Wang }

\date{}
\maketitle

\begin{abstract}
A Willmore surface $y:M\rightarrow S^{n+2}$ has a natural harmonic oriented conformal Gauss map $Gr_y:M\rightarrow SO^{+}(1,n+3)/SO(1,3)\times SO(n)$, which maps each point $p\in M$ to its oriented mean curvature 2-sphere at $p$.
An easy observation shows that all conformal Gauss maps  of Willmore surfaces satisfy a restricted nilpotency condition which will be called ``strongly conformally harmonic."
The goal of this paper is to characterize those strongly conformally harmonic maps
from a Riemann surface $M$ to $SO^+ (1, n + 3)/{SO^+(1, 3) \times SO(n) }$
which are the conformal Gauss maps of some Willmore surface in $S^{n+2}.$
It turns out that generically the condition of being strongly conformally harmonic suffices to
be associated to a Willmore surface. The exceptional case will also be discussed.
\end{abstract}
\vspace{1mm}  {\bf \ \ ~~Keywords:}  Willmore surfaces;  conformal Gauss map; S-Willmore surfaces.     \vspace{2mm}

{\bf\   ~~ MSC(2010): \hspace{2mm} 53A30, 53C30, 53C35}

%\tableofcontents

\section{Introduction}
One of the most classical topics of differential geometry is the investigation of specific classes of surfaces. Examples include Lagrange, Weierstrass and Riemann's works on on minimal surfaces, Hilbert's work on constant negative Gauss curvature surfaces and Hopf's work on global geometry of CMC surfaces. The study of such surfaces not only
produces many important results, but also leads to the development of new methods which have great influence both on geometry and on other fields of mathematics like analysis and PDE.

Since Gauss introduced the ``Gauss map" for surfaces,  the interaction between surfaces and their Gauss maps plays an important role in the study of surfaces, see e.g., \cite{Ruh-V}. Kenmotsu's classical work on surfaces with prescribed non-vanishing (generally  also non-constant) mean curvature and  their Gauss map \cite{Kenmotsu} lead to a new direction in geometry. A different, but related topic was investigated in a  series of papers by Hoffman and Osserman \cite{HO1} on surfaces in $\R^n$, and similarly by Weiner \cite{Weiner,Weiner2}.

In the study of the conformal geometry of surfaces, the conformal Gauss map plays an important role \cite{Blaschke,Bryant1984,Ejiri1988, Rigoli1987}. So it is natural to what kind of maps can be the conformal Gauss map of a surface in $S^{n+2}$. This is in particular an important problem concerning Willmore surfaces.
A Willmore surface in $S^{n+2}$ is a critical surface of the Willmore functional
\[\int_M(|\vec{H}|^2-K+1)\dd M,\]
with $\vec{H}$  and $K$ being the mean curvature vector and the Gauss curvature respectively.
The Willmore functional can be considered to be the bending energy of a closed
surface, while the bending energy of a rod  has elastic curves as solutions to the variational problem. The latter situation was investigated  by Germain already in 1821.

 It is well-known that the Willmore functional is invariant under conformal transformations of $S^{n+2},$   \cite{Blaschke}, which makes the study of this problem more difficult. The study of  the Willmore functional and of Willmore surfaces has  lead to  progress in several directions of geometry and analysis.   For example, Li-Yau \cite{LY} introduced the concept of ``conformal volume"  and gave a partial proof of the Willmore conjecture.  Later, L. Simon \cite{Simon} and Marques-Neves \cite{Marques} applied geometric measure theory   to study the Willmore functional, which lead to Marques-Neves' proof of  the Willmore conjecture in $S^3$ (\cite{Simon,Marques}). Analytical methods developed e.g. by Kuwert-Sch\"{a}tzle \cite{Ku-Sch}, Rivi\`{e}re \cite{Rivière}, for the discussion of the Willmore functional and Willmore surfaces also make important contributions to the study of Willmore surfaces as well as the field of geometric analysis and PDE.
  These examples show how the study of Willmore surfaces influences the development of geometry and analysis, which also explains our interest in this topic.

  Due to the work of Blaschke \cite{Blaschke}, Bryant \cite{Bryant1984}, Ejiri \cite{Ejiri1988} and Rigoli \cite{Rigoli1987}, a  Ruh-Villms type classical result was established, namely that a conformal immersion is Willmore if and only if its conformal Gauss map is harmonic.
Concerning Willmore surfaces, it is a natural and interesting geometric question of what kind of harmonic maps can be realized as the conformal Gauss map of some Willmore surface. This is in fact an essential problem when one wants to use the loop group method of integrable systems theory \cite{DPW}, since this method studies  Willmore surfaces in terms of conformal Gauss maps.
In spite of much  work on Willmore surfaces, independent of the work on the Willmore conjecture,  and from many different points of view (see, e.g., \cite{Be-Ri,Bryant1984,Ejiri1988,Helein,Ku-Sch,Marques,Rivière}), the question just stated above has  not been solved.

In this note we give an answer to this question.

A surface $y$ in $S^{n+2}$  has a mean curvature 2-sphere at each point. The conformal Gauss map $Gr_y$ maps each point to the oriented mean curvature 2-sphere at this point. Since in conformal geometry an oriented mean curvature 2-sphere of $S^{n+2}$  is identified with an oriented 4-dimensional Lorentzian subspace of the oriented $n+4$ dimensional Lorentz-Minkowski space $\R^{n+4}_1$ (See for example, \cite{BPP,Her1}),
$Gr_y$ can be also viewed as a map into $Gr_{3,1}(\R^{n+4}_{1})=SO^+(1,n+3)/SO^+(1,3)\times SO(n)$.
Here by $SO^+(1,n+3)$ we denote the connected component of the special linear isometry group of $\mathbb{R}^{n+4}_1$
which contains the identity element. Here $``\ ^+"$ comes from  the fact that  $SO^+(1,n+3)$ preserves
the forward light cone.

To obtain the characterization of all harmonic maps which are the conformal Gauss map of some Willmore surface, we have two things to do:

 \begin{enumerate}
\item  {\it From a Willmore surface to its  harmonic conformal Gauss map (The simple direction).}
We first derive a description of the Maurer-Cartan form $\alpha$ of a natural
frame  associated with the conformal Gauss map of a conformal surface
in $S^{n+2}$. This yields a very specific structure of the matrices occurring (see Theorem \ref{frame}). We also show that, conversely, given a  conformal map into $SO^+(1,n+3)/SO^+(1,3)\times SO(n)$ which has a frame whose Maurer-Cartan form has the mentioned specific form, then it is the conformal Gauss map of a conformal immersion. This result seems to be new to the best of the knowledge of these authors.

Moreover, from Theorem \ref{frame}, one obtains  $B_1^tI_{1,3}B_1=0$, or equivalently, the linear map $\mathfrak{B}$ satisfies $\mathfrak{B}^2|_{V^{\perp}} \equiv 0,$ where $\mathfrak{B}(X) =  F \alpha_{\mathfrak{p}}'(\frac{\partial}{\partial z})F^{-1}X$ (See  \eqref{eq-B} for more details).
From this it is already clear that only special conformally  harmonic maps can be the
conformal Gauss maps of Willmore surfaces. These special conformally harmonic maps will be called ``strongly conformally harmonic maps". The equation   $B_1^tI_{1,3}B_1=0$ and the equation  $\mathfrak{B}^2|_{V^{\perp}} \equiv 0$ both  use in some sense coordinates. These equations can be derived from some coordinate free expression $\mathcal{B}(X,\tilde{X})=F \alpha_{\mathfrak{p}}(\tilde{X})F^{-1}X$, where $\tilde{X}\in\Gamma(TM)$.
 Note that $\mathcal{B}$ can be viewed as the generalized Weingarten map \cite{LW}. For our purposes it will be most convenient to use coordinates and frames to carry out concrete computations. Note that  frames are the
 crucial tool of the DPW method which will be used to construct explicit examples and to characterize special surfaces. \footnote{
``We  share a philosophy about linear algebra: we think basis-free, we write basis-free, but when the chips are down we close the office door and compute with matrices like fury."
Irving Kaplansky, Page 88 of the book ``Paul Halmos: Celebrating 50 Years of Mathematics."
}

\item

{\it From a harmonic map to a Willmore surface, where it is possible (The difficult direction).}
In this procedure one starts from some harmonic map and choses a moving frame.
If this frame  has a  Maurer-Cartan  which has the special form  described
in Theorem \ref{frame}, then this theorem already ensures that the given map is  the conformal Gauss map of some conformal immersion.
But there are many different frames (all gauge equivalent though)   and therefore
the form of the Maurer-Cartan form depends on the choice of frame (i.e. gauge).
In particular, the special form of the Maurer-Cartan form as stated in  Theorem \ref{frame}
is not a gauge invariant criterion on the map, which makes it difficult to be checked and the geometric meaning is unclear. Further observation makes us realize that the
restricted nilpotency condition $\mathfrak{B}^2|_{V^{\perp}}\equiv0$ plays an important role. For a harmonic map $f$ , whose frame satisfies the restricted nilpotency condition we show in Theorem \ref{th-Willmore-harmonic-U} that it always induces  a map $y :U \rightarrow S^{n+2}$ such that, up to a change of orientation, either $y$ is a  Willmore surface  on an open dense subset of $M$ and $f$ its oriented conformal Gauss map, or $y$ degenerates to a point.\footnote{Note that a Willmore surface is conformally equivalent to a minimal surface in $\R^n$ if and only if it has a dual surface which degenerates to a point \cite{Bryant1984,Ejiri1988}. Hence the corresponding conformal Gauss map contains two maps with one degenerating to a point. Recall that a dual surface is another conformal map sharing the same conformal Gauss map as the surface \cite{Bryant1984,Ejiri1988,Ma2005,Ma2006}. In general a Willmore surface may have no dual surface \cite{Ejiri1988,BFLPP,Ma2006}. So it could happen that the ``conformal Gauss map" contains only one map which degenerates to a point.}
For the case that $y$ is non-degenerate on an open dense set, we show that if $f$ is the oriented conformal Gauss map of $y$, then this already implies the unique global existence (See Theorem \ref{th-Willmore-harmonic} for more details).
 \end{enumerate}

This paper is organized as follows: In Section 2 we recall the moving frame treatment of Willmore surfaces, following the method of \cite{BPP}, relating a Willmore surface to its conformal Gauss map. We also briefly compare our treatment with H\'{e}lein's framework \cite{Helein}. Then we introduce the basic facts about harmonic maps and apply them, in Section 3,  to describe the conformal Gauss maps (See Theorem \ref{th-Willmore-harmonic-U} and Theorem \ref{th-Willmore-harmonic}).
Appendix ends this paper with a technical proof of  Theorem \ref{normalizationlemma}.

%%%%%%%%%%%%%%%%%%%%%%%%%%%%%
%%%%%%%%%%%%%%%%%%%%%%%%%%%

\section{Conformal  surfaces in $S^{n+2}$}

 We will use the elegant treatment of the conformal geometry of surfaces in $S^{n+2}$ presented in \cite{BPP} and then describe  these surfaces  by using the Maurer-Cartan form of some lift.
%MMMMMMMMMMMMMMMMMMMMMMMMMMMMMMMMMMMMMMMMMMMMMMMMMMMMMMMMMMM

\subsection{Conformal surface theory in the projective light cone model}

Let $\mathbb{R}^{n+4}_1$ denote Minkowski space, i.e. we consider $\mathbb{R}^{n+4}$ equipped with the Lorentzian metric
$$<x,y>=-x_{0}y_0+\sum_{j=1}^{n+3}x_jy_j=x^t I_{1,n+3} y,\ \hspace{3mm}  I_{1,n+3}=diag(-1,1,\cdots,1).$$
Let $\mathcal{C}_+^{n+3}= \lbrace x \in \mathbb{R}^{n+4}_{1} |<x,x>=0 , x_0 >0 \rbrace $
denote the forward light cone of $\mathbb{R}^{n+4}_{1}$.
It is easy to see that the projective light cone
$
Q^{n+2}=\{\ [x]\in\mathbb{R}P^{n+3}\ |\ x\in \mathcal{C}_+^{n+3}
\}$
with the induced conformal metric, is conformally equivalent to $S^{n+2}$.
Moreover, the conformal group of
$Q^{n+2}$ is exactly the projectivized orthogonal group $O(1,n+3)/\{\pm1\}$ of
$\mathbb{R}^{n+4}_1$, acting on $Q^{n+2}$ by
$
T([x])=[Tx],\,\,\ T\in O(1, n+3).
$

Let $y:M\rightarrow S^{n+2}$ be a conformal immersion from a Riemann surface $M$.
 Let $U\subset M$ be a contractible open subset. A local
lift of $y$ is a map $Y:U\rightarrow \mathcal{C}_+^{n+3} $ such
that $\pi\circ Y=y$. Two different local lifts differ by a scaling,
thus they induce the same conformal metric on $M$.
Here we call $y$ a {\em conformal} immersion, if
$\langle Y_{z},Y_{z} \rangle =0$ and
$\langle Y_{z},Y_{\bar{z}} \rangle >0$ for any local lift $Y$ and any complex
coordinate $z$ on $M$.
Noticing
$\langle Y,Y_{z\bar{z}}\rangle =-\langle Y_{z},Y_{\bar{z}}\rangle <0$, we see that
\begin{equation}
V={\rm Span}_{\mathbb{R}}\{Y,{\rm Re}Y_{z},{\rm Im}Y_{z},Y_{z\bar{z}}\}
\end{equation}
is an oriented rank-4 Lorentzian sub-bundle over $U$, and
there is a natural decomposition
of the oriented trivial bundle $U\times \mathbb{R}^{n+4}_{1}=V\oplus V^{\perp}$, where $V^{\perp}$ is the orthogonal complement of $V$ with an induced natural orientation.
Note that both, $V$ and $V^{\perp}$,  are independent of the choice of $Y$
and $z$, and therefore are conformally invariant. In fact, we obtain a global conformally invariant bundle decomposition
$M\times \mathbb{R}^{n+4}_{1}=V\oplus V^{\perp}$. For any $p\in M$, we denote by $V_p$ the fiber of $V$ at $p$. And
the complexifications of  $V$ and $V^{\perp}$ are denoted by $V_{\mathbb{C}}$ and
$V^{\perp}_{\mathbb{C}}$ respectively.
Since $Y$ takes values in the forward light cone $\mathcal{C}_+^{n+3}$, we focus on conformal transformations which are contained in  $SO^+(1,n+3)$.

Fixing a local coordinate $z$ on $U$, there exists a unique local lift $Y$ in $\mathcal{C}^{n+3}$ satisfying
$|{\rm d}Y|^2=|{\rm d}z|^2$, i.e., $\langle Y_{z},Y_{\bar{z}}\rangle = \frac{1}{2}$. Such lift $Y$ is called the canonical lift with respect
to $z$.  Given a canonical lift $Y$ we choose the frame $\{Y,Y_{z},Y_{\bar{z}},N\}$ of
$V_{\mathbb{C}}$, where $N$ is the uniquely determined section of $V$ over $U$ satisfying
\begin{equation}\label{eq-N}
\langle N,Y_{z}\rangle=\langle N,Y_{\bar{z}}\rangle=\langle
N,N\rangle=0,\langle N,Y\rangle=-1.
\end{equation}
Note that $N$ lies in the forward light cone $\mathcal{C}_+^{n+3}$
 and that
$N\equiv 2Y_{z\bar{z}}\!\!\mod Y$ holds.

Next we define \emph{the conformal Gauss map} of $y$.

\begin{definition} \label{def-gauss} $($\cite{Bryant1984,BPP,Ejiri1988,Ma}$)$
Let $y:M\to S^{n+2}$ be a conformal immersion from a Riemann surface $M$. The  \emph{conformal Gauss map} of $y$ is defined by
\begin{equation}\begin{array}{ccccc}
                 Gr : & M &\rightarrow&
Gr_{1,3}(\mathbb{R}^{n+4}_{1}) &= SO^+(1,n+3)/SO^+(1,3)\times SO(n)\\
                \ & p\in M & \mapsto & V_p &\ \\
                \end{array}
\end{equation}
Here $V_p$ is  the $4-$dimensional Lorentzian subspace oriented by a basis $\{Y,N,Y_u,Y_v\}$.
\end{definition}

The orientation of $V$ implies that $Gr$ maps to  the symmetric space $SO^+(1,n+3)/SO^+(1,3)\times SO(n)$ instead of $SO^+(1,n+3)/S(O^+(1,3)\times O(n))$, which is usually used in the literature. Note also that $Gr$ depends on the conformal immersion $y$ as well as the chosen complex structure of the Riemann surface $M$.
 Note that $Gr$ depends on the conformal immersion $y$ as well as the chosen complex structure of the Riemann surface $M$. We will therefore sometimes write $Gr = Gr_y$ to emphasize this.

Given a (local) canonical lift $Y$ we note that $Y_{zz}$ is orthogonal to
$Y$, $Y_{z}$ and $Y_{\bar{z}}$. Therefore there exists a complex valued function $s$
and a section $\kappa\in \Gamma(V_{\mathbb{C}}^{\perp})$ such that
\begin{equation}
Y_{zz}=-\frac{s}{2}Y+\kappa.
\end{equation}
This defines two basic invariants of $y$:
$\kappa$, \emph{the conformal Hopf differential} of $y$ ,
and $s$, \emph{the Schwarzian} of $y$.
Clearly, $\kappa$ and $s$ depend on the
coordinate $z$ ( see \cite{BPP,Ma}).
Let $D$
denote the $V_{\mathbb{C}}^{\perp}$ part of the natural connection of
$\mathbb{C}^{n+4}$. Then for any section $\psi\in
\Gamma(V_{\mathbb{C}}^{\perp})$ of the normal bundle $V_{\mathbb{C}}^{\perp}$
and any (local) canonical lift $Y$ of some conformal immersion $y$ into $S^{n+2}$
we obtain
the structure equations (\cite{BPP}, \cite{Ma2006}):
\begin{equation}\label{eq-moving}
\left\{\begin {array}{lllll}
Y_{zz}=-\frac{s}{2}Y+\kappa,\\
Y_{z\bar{z}}=-\langle \kappa,\bar\kappa\rangle Y+\frac{1}{2}N,\\
N_{z}=-2\langle \kappa,\bar\kappa\rangle Y_{z}-sY_{\bar{z}}+2D_{\bar{z}}\kappa,\\
\psi_{z}=D_{z}\psi+2\langle \psi,D_{\bar{z}}\kappa\rangle Y-2\langle
\psi,\kappa\rangle Y_{\bar{z}}.
\end {array}\right.
\end{equation}
For these structure equations the integrability conditions are
the conformal Gauss, Codazzi and Ricci equations respectively (\cite{BPP}, \cite{Ma2006}):
\begin{equation}\label{eq-integ}
\left\{\begin {array}{lllll} \frac{1}{2}s_{\bar{z}}=3\langle
\kappa,D_z\bar\kappa\rangle +\langle D_z\kappa,\bar\kappa\rangle,\\
{\rm Im}(D_{\bar{z}}D_{\bar{z}}\kappa+\frac{\bar{s}}{2}\kappa)=0,\\
R^{D}_{\bar{z}z}=D_{\bar{z}}D_{z}\psi-D_{z}D_{\bar{z}}\psi =
2\langle \psi,\kappa\rangle\bar{\kappa}- 2\langle
\psi,\bar{\kappa}\rangle\kappa.
\end {array}\right.
\end{equation}

Choosing an oriented orthonormal frame $\{\psi_j,j=1,\cdots,n\}$ of the normal bundle $V^{\perp}$ over $U$,
we can write the normal connection as $D_z\psi_j=\sum_{l=1}^{n}b_{jl}\psi_l$ with $b_{jl}+b_{lj}=0.$
Then, the conformal Hopf differential $\kappa$ and its derivative $D_{\bar{z}}\kappa$ are of the form
\begin{equation} \label{defkappabeta}
\kappa=\sum_{j=1}^{n}k_j\psi_j,\ D_{\bar{z}}\kappa=\sum_{j=1}^{n}\beta_j\psi_j,\
\hbox{with }\beta_j=k_{j\bar{z}}- \sum_{j=1}^{n}\bar{b}_{jl}k_l,\ j=1,\cdots,n.\end{equation}
Finally, put
$\phi_1=\frac{1}{\sqrt{2}}(Y+N),\ \phi_2=\frac{1}{\sqrt{2}}(-Y+N),\ \phi_3= Y_z+Y_{\bar{z}},\ \phi_4=i(Y_z-Y_{\bar{z}}),\   k=\sqrt{\sum_{j=1}^{n}|k_j|^2},$
and set
\begin{equation}\label{F}
F:=\left(\phi_1,\phi_2,\phi_3,\phi_4,\psi_1,\cdots,\psi_n\right).
\end{equation}
Note that $F$ is a frame of the conformal Gauss map $Gr$. From the above computation, we obtain directly (Here $I_{1,3}=\hbox{diag}(-1,1,1,1)$)
\begin{theorem}\label{frame}\
\begin{enumerate}
\item Let $y:M\rightarrow  S^{n+2}$ be a conformal immersion
and $Y$ its canonical lift over the open contractible set $U \subset M$.
Then the frame $F$ attains values in $SO^+(1,n+3)$,
and  the Maurer-Cartan form $\alpha=F^{-1}\dd F$ of $F$ is of the form \[\alpha=\left(
                   \begin{array}{cc}
                     A_1 & B_1 \\
                     B_2 & A_2 \\
                   \end{array}
                 \right)\dd z+\left(
                   \begin{array}{cc}
                     \bar{A}_1 & \bar{B}_1 \\
                     \bar{B}_2 & \bar{A}_2 \\
                   \end{array}
                 \right)\dd \bar{z},\]
with
\begin{equation}A_1=\left(
                             \begin{array}{cccc}
                               0 & 0 & s_1 & s_2\\
                               0 & 0 & s_3 & s_4 \\
                               s_1 & -s_3 & 0 & 0 \\
                               s_2 & -s_4 & 0 & 0 \\
                             \end{array}
                           \right),\   A_2=\left(
                             \begin{array}{cccc}
                               b_{11} & \cdots &  b_{n1} \\
                               \vdots& \vdots & \vdots \\
                               b_{1n} &\cdots & b_{nn} \\
                             \end{array}
                           \right),\ \end{equation}
\begin{equation}\label{s}
\left\{\begin{split}&s_1=\frac{1}{2\sqrt{2}}(1-s-2k^2),\ s_2=-\frac{i}{2\sqrt{2}}(1+s-2k^2),\\
&s_3=\frac{1}{2\sqrt{2}}(1+s+2k^2),\ s_4=-\frac{i}{2\sqrt{2}}(1-s+2k^2),\\
\end{split}\right.
\end{equation}
\begin{equation} \label{B1}
B_1=\left(
      \begin{array}{ccc}
         \sqrt{2} \beta_1 & \cdots & \sqrt{2}\beta_n \\
         -\sqrt{2} \beta_1 & \cdots & -\sqrt{2}\beta_n \\
        -k_1 & \cdots & -k_n \\
        -ik_1 & \cdots & -ik_n \\
      \end{array}
    \right),  \ \
B_2=\left(
      \begin{array}{cccc}
        \sqrt{2} \beta_1& \sqrt{2} \beta_1 & k_1& i k_1 \\
        \vdots & \vdots & \vdots & \vdots\\
      \sqrt{2} \beta_n& \sqrt{2} \beta_n & k_n & ik_n \\
      \end{array}
    \right)=-B_1^tI_{1,3}.
 \end{equation}
                       \begin{equation*} \end{equation*}

\item Conversely, assume we have a conformal map $f:U\rightarrow SO^+(1,n+3)/SO^+(1,3)\times SO(n)$. If it has some frame $F=(\phi_1,\cdots,\phi_4,\psi_1,\cdots,\psi_{n+4}):U\rightarrow SO^+(1,n+3)$
such that the Maurer-Cartan form $\alpha=F^{-1}\dd F$  of $F$ is of the above form,
then
\begin{equation}\label{eq-y-from-map}y=\pi_0(F)=:\left[ (\phi_1-\phi_2)\right]
\end{equation}
is a conformal immersion from $U$ into $S^{n+2}$ with $f$ being its conformal Gauss map.\end{enumerate}
    \end{theorem}
\begin{remark}
 \begin{enumerate}
\item There are two conditions in the above theorem for $f$ to be the conformal Gauss map of some surface $y$, i.e., $B_1$ being a special form and $A_1$ being a special form at the same time. For Willmore surfaces, or equivalently, for $f$ being harmonic, we will show that the restriction on $B_1$ is enough, except the degenerate case.
\item
As we have seen in the Theorem just above, the form of the Maurer-Cartan characterizes conformal maps and therefore this form is  of great importance.
 In particular,
 \[B_1 ^t I_{1,3} B_1 = 0\]
 holds. From this one see that for any point in $M$ the rank of $B_1$ is at most 2. The case $B_1\equiv 0$ is equivalent to $y$ being conformally equivalent to a round sphere. For the non-trivial case, $B_1 \neq 0$, a detailed discussion will be given in Theorem \ref{th-Willmore-harmonic-U}.
 \end{enumerate}

\end{remark}

\begin{definition} For a conformal map $f:M\rightarrow SO^+(1,n+3)/SO^+(1,3)\times SO(n)$,
We define the linear map $\mathfrak{B}$ as follows:
\begin{equation}\label{eq-B}
 \begin{array}{cccc}
  \mathfrak{B}:&\R^{n+4}_1& \rightarrow &\R^{n+4}_1\\
 \ & X &\mapsto& F \alpha_{\mathfrak{p}}'(\frac{\partial}{\partial z})F^{-1}X
 \end{array}
\end{equation}
It is straightforward to see that $\mathfrak{B}$ is well-defined, i.e., independent of the choice of the  frame  $F$.
\end{definition}
\begin{remark}\
\begin{enumerate}
\item Note that $\mathfrak{B}(V_{\C})=\partial f$, if we use the notion $\partial f$ as Chern and Wolfson defined in \cite{Chern-W}.
\item In \cite{Hitchin}, Hitchin  introduced the idea of decomposing $\dd$ into the connection part and a nilpotent linear map. Here $\mathfrak{B}$ is the linear map. Note that the same idea is also used in the definition of $\beta$ in \cite{BR} if one uses $\beta$ acting on a section as in Proposition 1.1 of \cite{BR}.
\end{enumerate}
\end{remark}
Using the fact that
$\alpha_{\mathfrak{p}}'(\frac{\partial}{\partial z})=\left(
                   \begin{array}{cc}
                     0 & B_1 \\
                     -B_1^tI_{1,3} & 0 \\
                   \end{array}
                 \right),$
                 we obtain immediately
\begin{corollary}\
 \begin{enumerate}
\item $\mathfrak{B}(V_{\C})\subset V_{\C}^{\perp}$, $\mathfrak{B}(V_{\C}^{\perp})\subset V_{\C}$.
\item The linear map $\mathfrak{B}$ satisfies the ``restricted nilpotency condition"
\begin{equation}\label{eq-nilpotent-B}
\mathfrak{B}^2|_{V^\perp}=0.
\end{equation}
if and only if the Maurer-Carten form of $f$ satisfies
\[B_1^t I_{1,3}B_1=0.\]
\end{enumerate}
\end{corollary}

%F0831
%MMMMMMMMMMMMMMMMMMMMMMMMMMMMMMMMMMMMMMMMMMMMMMMMMMMMMMMMMMMMMMMMMMMMMMMMMMM

\subsection{Willmore surfaces and harmonicity}

The conformal
Hopf differential $\kappa$ plays an important role in the investigation of Willmore
surfaces.
A direct computation using \eqref{eq-moving} shows that the conformal Gauss map $Gr$ induces a conformally invariant (possibly degenerate )
metric
\[
g:=\frac{1}{4}\langle {\rm d}G,{\rm d}G\rangle=\langle
\kappa,\bar{\kappa}\rangle|\dd z|^{2}
\]
globally on $M$ (see \cite{BPP}). Note that this metric degenerates at umbilical points of $y$, which are by definition the points where $\kappa$ vanishes.

\begin{definition}(\cite{BPP}, \cite{Ma2006}) The \emph{ Willmore functional} of $y$ is
defined as four  times the area of M with respect to the metric above:
\begin{equation}\label{eq-W-energy}
W(y):=2i\int_{M}\langle \kappa,\bar{\kappa}\rangle \dd z\wedge
\dd \bar{z}.
\end{equation}
An immersed surface $y:M\rightarrow S^{n+2}$ is called a
\emph{Willmore surface}, if it is a critical point of the Willmore
functional with respect to any variation (with compact support) of $y$.
\end{definition}

Note that the above definition of the Willmore functional coincides
with the usual definition (See \cite{BPP,Ma-W1}). It is well-known that Willmore surfaces are characterized as
follows
\cite{Bryant1984,BPP,Ejiri1988,Rigoli1987,Wang1998}.

\begin{theorem}\label{thm-willmore} For a conformal immersion $y:M\rightarrow  S^{n+2}$, the following three conditions
are equivalent:
 \begin{enumerate}
\item $y$ is Willmore;

\item The conformal Gauss map $Gr_y$ is a conformally harmonic map into
$G_{3,1}(\mathbb{R}^{n+3}_{1})$;

\item The conformal Hopf differential $\kappa$ of $y$ satisfies the
``Willmore condition":
\begin{equation}\label{eq-willmore}
D_{\bar{z}}D_{\bar{z}}\kappa+\frac{\bar{s}}{2}\kappa=0
\end{equation}
for any contractible chart of $M$.
 \end{enumerate}
\end{theorem}

Now we introduce the notion of the so-called ``dual Willmore surface", which is of essential importance in Bryant's and Ejiri's description of Willmore two-spheres.
\begin{definition}\label{def-dual} (\cite{Bryant1984}, page 399 of \cite{Ejiri1988})
Let $y:M\rightarrow S^{n+2}$ be a Willmore surface with $M_0$ the set of umbilical points and with $Gr$  its conformal Gauss map. A conformal map $\hat{y}:M\setminus M_0\rightarrow S^{n+2}$ which does not coincide with $y$ is called a ``dual surface" of $y$, if either $\hat y$ reduces to a point or on an open dense subset of $M\setminus M_0$ the map  $\hat y$  is an immersion  and the conformal Gauss map $Gr_{\hat{y}}$ of $\hat{y}$ spans at all points the same subspace as $Gr_y$ (Clearly,  $\hat{y}$  then also is a Willmore surface).
\end{definition}
  There exist many Willmore surfaces (\cite{Bryant1984, Ejiri1988,Ma2005}) which admit dual Willmore surfaces.
In general a Willmore surface may not admit a dual surface.
 For the discussion of Willmore surfaces admitting a dual surface,
Ejiri \cite{Ejiri1988} introduced the so-called {\em S-Willmore surfaces}.
 It is convenient to define S-Willmore surfaces here as follows (see also \cite{Ma2005}):
\begin{definition} (\cite{Ejiri1988})
A Willmore immersion $y:M\rightarrow S^{n+2}$ is called an S-Willmore surface if on any open subset $U$, away from the umbilical points, the conformal Hopf differential $\kappa$ of $y$ satisfies
$D_{\bar{z}}\kappa || \kappa,\
~\hbox{ i.e.  }\  D_{\bar{z}}\kappa+\frac{\bar{\mu}}{2}\kappa=0~\hbox{ for some } \mu:U\rightarrow \mathbb{C}. $
\end{definition}

\begin{corollary}\label{S-frame} Let $y$ be a Willmore surface which is not totally umbilical. Then $y$ is S-Willmore  if and only if the
(maximal) rank of $B_1$ in Theorem \ref{frame} is $1$.
\end{corollary}

\begin{theorem}\label{thm-dual gauss map} \
 \begin{enumerate}
 \item $($\cite{Bryant1984}, Theorem 7.1 of \cite{Ejiri1988}$)$ A (non totally umbilical) Willmore surface $y$ is S-Willmore  if and only if it has a unique dual (Willmore) surface $\hat y$ on $M\setminus M_0$.
Moreover, if $y$ is S-Willmore, the dual map $\hat y$ can be extended to $M$.
\item $($Theorem 2.9 of \cite{Ma}$)$ If the dual surface $\hat y$ of $y$ is immersed at $p\in M$,  then $Gr_{\hat{y}}(p)$ spans the same subspace as $Gr_y(p)$, but its orientation  is  opposite to the one of $Gr_y(p)$.
\item  (\cite{Ejiri1988},\cite{Ma}) If two Willmore surfaces  $y$ and $\tilde y$ share the same {\em oriented} conformal Gauss map, then $y=\tilde y$.
 \end{enumerate}
\end{theorem}
\begin{proof} The result of  (3) is not presented explicitly in \cite{Ejiri1988,Ma}. So we give a proof as follows. If $y$ is non-S-Willmore, then the mean curvature two-sphere congruence of $y$ envelops only one surface $y$ by  Lemma 1.3 and Lemma 3.1 of \cite{Ejiri1988}, or Proposition 2.12 of \cite{Ma}. Since  $y$ and $\tilde y$ share the same conformal Gauss map, we have $y=\tilde y$. If $y$ is non-S-Willmore, then (2) tells $y=\tilde y$.
\end{proof}

We will say ``the conformal Gauss map contains a constant lightlike vector $Y_0 $'' if there exists a non-zero constant lightlike vector $Y_0$ in $\mathbb{R}^{n+4}_1$ satisfying $Y_0\in V_{p}$ for all $p\in M$.  Then a well-known fact states (one can find a proof on the bottom of page 1573 of \cite{Ma-W1} (similar ideas are also used on pages 378-379 of \cite{Helein}).
\begin{theorem}\label{minimal} A Willmore surface $y$ is conformally equivalent to a minimal surface in $R^{n+2}$  if and
only if its conformal Gauss map $Gr$ contains a constant lightlike vector.
\end{theorem}

There exist Willmore surfaces which fail to be immersions at some points. To include surfaces of this type, we introduce the notion of {\em Willmore maps} and {\em strong Willmore maps}.
\begin{definition}
A smooth map $y$ from a Riemann surface $M$ to $S^{n+2}$ is called a Willmore map if it is a conformal Willmore immersion on an open dense subset $\hat{M}$ of $M$.  If $\hat{M}$ is maximal, then the points in $M_0=M\backslash \hat{M}$ are called branch points of $y$, at which points $y$ fails to be an immersion. Moreover, $y$ is called a strong Willmore map if it is a Willmore map and if the conformal Gauss map $Gr: \hat{M}\rightarrow SO^+(1,n+3)/SO^+(1,3)\times SO(n)$ of $y$ can be extended smoothly (and hence real analytically) to $M$.
\end{definition}
\begin{remark}
It is an interesting (open) problem under which conditions a  Willmore map will be a strong Willmore map.
\end{remark}

\begin{example}\
 \begin{enumerate}
 \item Let $x: \hat{M}\rightarrow \mathbb{R}^n$ be a complete minimal surface with finitely many ends $\{p_1,\cdots, p_r\}$. By the inverse  of the stereographic projection $x$ becomes a smooth map $y$ from a compact Riemann surface $M=\hat{M}\cup\{p_1,\cdots, p_r\}$ to $S^n$. If all the ends of $x$ are embedded planar ends, $y$ will be a Willmore immersion (\cite{Bryant1984}). If some planar ends $\{p_{j_1},\cdots, p_{j_t}\}$ fail to be embedded, $y$ will be a strong Willmore map with branch  points $\{p_{j_1},\cdots, p_{j_t}\}$. If some ends $\{\tilde{p}_{1},\cdots,\tilde{p}_{l}\}$ fail to be planar, $y$ will be a Willmore map with branch  points $\{\tilde{p}_{1},\cdots,\tilde{p}_{l}\}$ with its conformal Gauss map having no definition on these ends.

\item It is well-known that all minimal surfaces in Riemannian space forms
 can be considered to be  Willmore surfaces in some $S^m$ (\cite{Bryant1984,Weiner}).
These surfaces are basic examples of Willmore surfaces.
Moreover, they are S-Willmore surfaces, see \cite{Ejiri1988,Ma}.

\item The first non-minimal Willmore surface was given by Ejiri in \cite{Ejiri1982}.
This non-S-Willmore Willmore surface is a homogeneous torus in $S^5$.
Later, using the Hopf bundle, Pinkall produced a family of non-minimal Willmore tori in $S^3$ via elastic curves (\cite{Pinkall1985}).\\
 \end{enumerate}
\end{example}

\begin{remark} ({\em H\'{e}lein and Ma's harmonic maps})

In \cite{Helein}, H\'{e}lein extended the treatment of Bryant \cite{Bryant1984} to deal with Willmore surfaces in $S^3$ by using
a loop group method \cite{DPW}. He used two kinds of harmonic maps: the conformal Gauss map and the ones he  called  ``roughly harmonic maps". In terms of the notation used here, for a Willmore immersion $y$ in $S^3$ with
a local lift $Y$,  choose $\hat{Y}\in \Gamma (V)$ such that $\langle \hat{Y}, \hat{Y}\rangle=0,$  and $\langle Y, \hat{Y}\rangle=-1.$ Then  H\'{e}lein's roughly harmonic map is defined by
\begin{equation}\mathfrak{H}=Y\wedge \hat{Y}: M\rightarrow Gr_{1,1}(R^{5}_1).
\end{equation}
The reason of the name ``roughly harmonic" is that although $\mathfrak{H}$ may not be harmonic in general, it really provides another family of flat connections (see (36) page 350 in \cite{Helein} for details). If one assumes furthermore that  $\hat{Y}$ satisfies
\begin{equation}\label{H-M} \hat{Y}_{z}\in \hbox{Span}_{\C}\{\hat{Y}, Y, Y_z\} \mod V^{\perp}_{\C} ~ \ \hbox{ for all } z,
\end{equation}
$\mathfrak{H}$ will be a harmonic map. Especially, for a Willmore surface $y$ in $S^3$, there always exists a dual surface (Recall Definition \ref{def-dual} or see \cite{Bryant1984}).  When $\hat{Y}$ is chosen as the lift of the dual surface $\hat{y}$ of $y$, one obtains an interesting harmonic map connecting the original surface and its dual surface.
It is straightforward to generalize H\'{e}lein's notion of  roughly harmonic maps to the case of  Willmore immersions into $S^{n+2}$, since the definition above does not involve the co-dimensional information.
Such natural generalizations following  H\'{e}lein have been worked out in \cite{Xia-Shen}, by using the treatment of \cite{Wang1998} on Willmore submanifolds.

In a different development, in \cite{Ma2006}, Ma considered the generalization of the notion of  a dual surface for a Willmore surface $y$ in $S^{n+2}$.  Let $\hat{Y}\in \Gamma (V)$ with $\langle \hat{Y},\hat{Y}\rangle=0,$ and $ \langle Y, \hat{Y}\rangle=-1.$
Ma found that if $\hat{Y}$ satisfies
\begin{equation} \begin{split}
&\hat{Y}_{z}\in \hbox{Span}_{\C}\{\hat{Y}, Y, Y_z\} \mod V^{\perp}_{\C}  ~~\hbox{and }~~ \langle \hat{Y}_{z},\hat{Y}_{z}\rangle=0~ \ \hbox{ for all } z,\
\end{split}
\end{equation}
then $[\hat{Y}]$ is a new Willmore surface (which may degenerate to a point, see \cite{Ma2006}).
In this case $[\hat{Y}]$ is called  ``an adjoint surface" of $y$. Different from dual surfaces, the adjoint surface $[\hat{Y}]$ is in general  not unique (a detailed discussion on this can be found in \cite{Ma2006}). Moreover, Ma showed that for an adjoint surface
 $[\hat{Y}]$ the map  $\mathfrak{H}=Y\wedge \hat{Y}: M\rightarrow Gr_{1,1}(R^{n+3}_1)$ is a conformal harmonic map.
Obviously this harmonic map defined by Ma is just a special case of H\'{e}lein's harmonic maps used in  \cite{Helein}, \cite{Xia-Shen}, but it may be a particularly natural case (See \cite{Helein}).

Note that  for H\'{e}lein's harmonic map as well as for Ma's adjoint surfaces, it is usually not possible to prove
 global existence, since the solution of the equation \eqref{H-M} may have singularities. To be concrete, first note that \eqref{H-M} is exactly the Riccati equation \[\mu_{z}-\frac{\mu^2}{2}-s=0.\]
Note that for S-Willmore surfaces, $\mu$ may take the value $\infty$ at some points. Therefore, using the expression of the dual surface (see \cite{Ma2006}), at the points where $\mu$ approaches $\infty$, we have $[\hat{Y}]=[Y]$. This implies that the 2-dimensional Lorentzian bundle $Span_{\R}\{Y, \hat{Y}\}$ defined by $Y$ and $\hat{Y}$ reduces to a 1-dimensional lightlike bundle at these points. It stays unknown how to deal with the global properties for this kind of harmonic maps by using H\'{e}lein's approach.
  This is one of the reasons why we use the conformal Gauss map to study Willmore
surfaces, although the computations using H\'{e}lein's harmonic map would perhaps be somewhat easier.
 this is exactly why we chose this approach.
\end{remark}

%MMMMMMMMMMMMMMMMMMMMMMMMMMMMMMMMMMMMMMMMMMMMMMMMMMMMMMMMMMMMMMMMMMMMMMMMMMMMM

\section{Conformally harmonic maps into $SO^+(1,n+3)/SO^+(1,3)\times SO(n)$}

In this section, we first review the basic description of harmonic maps.
Then we apply it to the harmonic maps into
 $SO^+(1,n+3)/SO^+(1,3)\times SO(n)$. Since not every such conformally harmonic map
is the conformal Gauss map of some  strong Willmore map, we give a
necessary and sufficient condition for a conformally harmonic map to be the conformal Gauss map of a strong Willmore map.

%MMMMMMMMMMMMMMMMMMMMMMMMMMMMMMMMMMMMMMMMMMMMMMMMMMMMMMMMMMMMMMMMMMMMMMMMMMMMMMMMMMMMMM

\subsection{Harmonic maps into the symmetric space G/K}

Let $N=G/K$ be a symmetric space with involution $\sigma: G\rightarrow G$ such
that $G^{\sigma}\supset K\supset(G^{\sigma})_0$. Let $\pi:G\rightarrow G/K$ denote the projection of $G$ onto $G/K$. Let $\mathfrak{g}$ and $\mathfrak{k}$ denote the
Lie algebras of $G$ and $K$ respectively. The involution $\sigma$ induces the Cartan decomposition
$\mathfrak{g}=\mathfrak{k}\oplus\mathfrak{p},$ with $ [\mathfrak{k},\mathfrak{k}]\subset\mathfrak{k},\
 [\mathfrak{k},\mathfrak{p}]\subset\mathfrak{p}$ and $[\mathfrak{p},\mathfrak{p}]\subset\mathfrak{k}.$

Let $f:M\rightarrow G/K$ be a conformally harmonic map from a connected Riemann surface $M$.
Let $U\subset M$ be an open contractible subset.
Then there exists a frame $F: U\rightarrow G$ such that $f=\pi\circ F$ on $U$.
Let $\alpha$ denote the Maurer-Cartan form of $F$. Then $\alpha$ satisfies the Maurer-Cartan equation
and altogether we have
$F^{-1}\dd F= \alpha,$ with $\dd \alpha+\frac{1}{2}[\alpha\wedge\alpha]=0.$
Decomposing $\alpha$ with respect to $\mathfrak{g}=\mathfrak{k}\oplus\mathfrak{p}$ we obtain
\begin{equation*}\alpha=\alpha_{ \mathfrak{k}  } +\alpha_{ \mathfrak{p} },\ \hbox{ with }\
\alpha_{\mathfrak{k  }}\in \Gamma(\mathfrak{k}\otimes T^*M),\
\alpha_{ \mathfrak{p }}\in \Gamma(\mathfrak{p}\otimes T^*M).
\end{equation*}
Next we decompose $\alpha_{\mathfrak{p}}$ further into the $(1,0)-$part $\alpha_{\mathfrak{p}}'$ and the $(0,1)-$part $\alpha_{\mathfrak{p}}''$,
and set  \begin{equation}\label{eq-harmonic-lam}
\alpha_{\lambda}=\lambda^{-1}\alpha_{\mathfrak{p}}'+\alpha_{\mathfrak{k}}+\lambda\alpha_{\mathfrak{p}}'', \hspace{5mm}  \lambda\in S^1.
\end{equation}

\begin{lemma}\label{lemma-harmonic} $($\cite{DPW}$)$ The map  $f:M\rightarrow G/K$ is harmonic if and only if
\begin{equation}\label{integr}d
\alpha_{\lambda}+\frac{1}{2}[\alpha_{\lambda}\wedge\alpha_{\lambda}]=0,\ \ \hbox{for all}\ \lambda \in S^1.
\end{equation}
\end{lemma}

\begin{definition} Let $f:M\rightarrow G/K$ be harmonic
and $\alpha_{\lambda}$  the differential one-form defined above. Since
$\alpha_{\lambda}$ satisfies the integrability condition \eqref{integr},  we consider
the equation
\[\left\{\begin{split}\dd F(z,\lambda)&= F(z, \lambda)\alpha_{\lambda},\\
F(z_0,\lambda)&=\mathbf e\\
\end{split}\right.
\]
on any contractible open subset $U \subset M$,
where $z_0$ is a fixed base point on $U$, and $\mathbf e$ is the identity element in $G$. The map $F(z, \lambda)$
is called the {\em extended frame}
 of the harmonic map $f$ normalized at the base point $z=z_0$. Note that $F$ satisfies $F(z,\lambda =1)=F(z)$.
 \end{definition}

Consider $TM^{\mathbb{C}}=T'M\oplus T''M$ and write $\dd=\partial+\bar{\partial}$.
Then Lemma \ref{lemma-harmonic} can be restated as
\begin{lemma} $($\cite{DPW}$)$ The map $f:M\rightarrow G/K$ is harmonic if and only if
\begin{equation}
\left\{\begin{array}{ll}
 & \dd\alpha_{\mathfrak{k}}+\frac{1}{2}[\alpha_{\mathfrak{k}}\wedge\alpha_{\mathfrak{k}}]+\frac{1}{2}[\alpha_{\mathfrak{p}}\wedge\alpha_{\mathfrak{p}}]=0,\\
& \bar{\partial}\alpha_{\mathfrak{p}}'+[\alpha_{\mathfrak{k}}\wedge\alpha_{\mathfrak{p}}']=0.
\end{array}\right. \end{equation}
\end{lemma}

%MMMMMMMMMMMMMMMMMMMMMMMMMMMMMMMMMMMMMMMMMMMMMMMMMMMMMMMMM

\subsection{Harmonic maps into $SO^+(1,n+3)/SO^+(1,3)\times SO(n)$}

 Let's consider again $\mathbb{R}^{n+4}_1$, with the metric introduced in Section 2 and  the group   $SO^+(1,n+3)$ together with its Lie algebra
\begin{equation}\label{eq-so(1,n+3)}
\mathfrak{so}(1,n+3)=\mathfrak{g}=\{X\in gl(n+4,\mathbb{R})|X^tI_{1,n+3}+I_{1,n+3}X=0\}.
\end{equation}
Here $I_{1,n+3}=\hbox{diag}(-1,1,\cdots,1).$

Consider the involution
 \begin{equation}
\begin{array}{ll}
\sigma:  SO^+(1,n+3)& \rightarrow SO^+(1,n+3)\\
 \ \ \ \ \ \ \ A&\mapsto D^{-1} A D,
\end{array}\end{equation}
with
 $$D=\left(
         \begin{array}{ccccc}
             -I_{4} & 0 \\
            0 & I_n \\
         \end{array}
       \right),
       $$
       where $I_k$ denotes the $k\times k$ identity matrix. Then the fixed point group $SO^+(1,n+3)^{\sigma}$ of $\sigma$ contains $SO^+(1,3)\times SO(n)$,  where
$SO^+(1,3)$ denotes the connected component of $SO(1,3)$ containing $I$.
Moreover we have $SO^+(1,n+3)^{\sigma}\supset SO^+(1,3)\times SO(n) = (SO^+(1,n+3)^{\sigma})^0$,
where the superscript $0$ denotes  the connected component containing the identity element. On the Lie algebra level we obtain
   \begin{equation*}\begin{split}&\mathfrak{g}=
\left\{\left(
                   \begin{array}{cc}
                     A_1 & B_1 \\
                     -B_1^tI_{1,3} & A_2 \\
                   \end{array}
                 \right)
 |\ A_1^tI_{1,3}+I_{1,3}A_1=0, \ \ A_2+A_2^t=0\right\},\\
&\mathfrak{k}=\left\{\left(
                   \begin{array}{cc}
                     A_1 &0 \\
                     0 & A_2 \\
                   \end{array}
                 \right)
 | \ A_1^tI_{1,3}+I_{1,3}A_1=0,\ \ A_2+A_2^t=0\right\},\ \mathfrak{p}=\left\{\left(
                   \begin{array}{cc}
                   0 & B_1 \\
                     -B_1^tI_{1,3} & 0 \\
                   \end{array}
                 \right)
\right\}.
\end{split}
\end{equation*}

Now let $f: M\rightarrow SO^+(1,n+3)/SO^+(1,3)\times SO(n)$
 be a harmonic map with local frame $F: U\rightarrow SO^+(1,n+3)$ and Maurer-Cartan form $\alpha$ on some contractible open subset $U$ of $M$.
Let $z$ be a local complex coordinate on $U$. Writing
\begin{equation}\label{eq-B0}
    \alpha_{\mathfrak{k}}'=\left(
                   \begin{array}{cc}
                     A_1 &0 \\
                     0 & A_2 \\
                   \end{array}
                 \right) \dd z,\ \hbox{ and }\ \alpha_{\mathfrak{p}}'=\left(
                   \begin{array}{cc}
                     0 & B_1 \\
                     -B_1^tI_{1,3} & 0 \\
                   \end{array}\right)\dd z,\end{equation}
the harmonic map equations can be rephrased equivalently in the form
\begin{equation} \label{harmonic}
\left\{\begin{split}
 & Im \left( A_{1\bar{z}} +\bar{A}_1A_1-\bar{B}_1B_1^tI_{1,3}\right)=0,\\
 & Im \left(  A_{2\bar{z}} +\bar{A}_2A_2-\bar{B}_1^tI_{1,3}B_1\right)=0,\\
 &  B_{1\bar{z}} +\bar{A}_1B_1-B_1\bar{A}_2=0.
\end{split}\right. \end{equation}

In Section 2 we have seen that the Maurer Cartan form of the frame associated with a
Willmore surface in $S^{n+2}$ has a very special form. Fortunately, it is easy to detect when such a special form can be obtained by gauging.
A crucial part of our paper is the following result.

\begin{theorem}\label{normalizationlemma} Let $B_1$, as in \eqref{eq-B0}, be part of the Maurer-Cartan form of some non-constant harmonic map which is defined on the open, contractible Riemann surface $U$.
Then
\begin{equation}\label{eq-B1-2} B_1^tI_{1,3}B_1=0
\end{equation}
if and only if
there exists  a real analytic map $\mathbb{A}:  U\rightarrow SO^+(1,3)$ such that
\begin{equation}\label{eq-B1-standard}
\mathbb{A}B_1=\left(
      \begin{array}{ccc}
        \sqrt{2}\beta_1 & \cdots &\sqrt{ 2}\beta_n \\
        -\sqrt{2}\beta_1 & \cdots &  -\sqrt{ 2}\beta_n \\
        -k_1 & \cdots & -k_n \\
        -ik_1 & \cdots & -ik_n \\
      \end{array}
    \right) ,\ \hbox{or }~~
    \mathbb{A}B_1=\left(
      \begin{array}{ccc}
        \sqrt{2}\beta_1 & \cdots &\sqrt{ 2}\beta_n \\
        -\sqrt{2}\beta_1 & \cdots &  -\sqrt{ 2}\beta_n \\
        -k_1 & \cdots & -k_n \\
        ik_1 & \cdots & ik_n \\
      \end{array}
    \right)\ \hbox{ on }    U.
\end{equation}
\end{theorem}
\begin{proof}See Appendix B.\end{proof}
\begin{remark}  Recall  from  \eqref{defkappabeta}  that $k_j\equiv 0$ for all $j$ on an open subset implies $\kappa = 0$ and hence the  surface and its whole associated family  is totally umbilical. In particular, such surfaces have $B_1\equiv0$ and thus describe surfaces conformally equivalent to a round sphere. Such surfaces are trivial and therefore will not be considered.
\end{remark}

\begin{lemma} Let $f: U\rightarrow SO^+(1,n+3)/SO^+(1,3)\times SO(n)$ be a non-constant harmonic map on a contractible open Riemann surface $U$
 with two frames
$F,\ \hat{F}:U \rightarrow SO^+(1,n+3)$ and Maurer-Cartan forms $\alpha,\hat{\alpha}$.
Using a local complex coordinate $z$ on $U$, we write
$$  \alpha_{\mathfrak{p}}'=\left(
                   \begin{array}{cc}
                     0 & B_1 \\
                     -B_1^tI_{1,3} & 0 \\
                   \end{array}\right)\dd z, \hspace{1cm} \hat\alpha_{\mathfrak{p}}'=\left(
                   \begin{array}{cc}
                     0 & \hat{B}_1 \\
                     -\hat{B}_1^tI_{1,3} & 0 \\
                   \end{array}\right)\dd z.$$
Then
\begin{enumerate}
\item $B_1^t I_{1,3} B_1 = 0$ if and only if $  \hat{B}_1^t I_{1,3} \hat{B}_1 = 0$;
\item $\bar B_1^t I_{1,3} B_1 = 0$ if and only if $  \bar{\hat{B}}_1^t I_{1,3} \hat{B}_1 = 0.$
\end{enumerate}
\end{lemma}

\begin{proof}
Since $F$ and $\hat{F}$ are lifts of the same harmonic map $f$, there exists $F_0=\hbox{diag}(   F_{01},F_{02}): U\rightarrow SO^+(1,3)\times SO(n)
$ such that $\hat{F}=F\cdot F_0$. Then
$\hat{\alpha}=F_0^{-1}\alpha F_0+F_0^{-1}dF_0$, yielding $\hat{B}_1=F_{01}^{-1}B_1F_{02}$. So
$\hat{B}_1^t I_{1,3} \hat{B}_1= F_{02}^{-1}B_1^tF_{01}^{-1,t}I_{1,3}F_{01}^{-1}B_1F_{02}= F_{02}^{-1}B_1^tI_{1,3}B_1F_{02}.$
The last statement comes from the fact that $F_{01}$ and $F_{02}$ are real matrices.
\end{proof}

\begin{definition} \label{stronglyconfharm}
Let $f: M\rightarrow SO^+(1,n+3)/SO^+(1,3)\times SO(n)$ be a harmonic map. We call $f$ a {\bf strongly conformally harmonic map} if for any point $p\in M$, there exists a neighborhood $U_p$ of $p$ and a frame $F$ (with Maurer-Cartan form  $\alpha$) of $y$ on $U_p$ satisfying
\begin{equation}\label{eq-Willmore harmonic} B_1^t I_{1,3} B_1 = 0, ~ \hbox{where }~ \alpha_{\mathfrak{p}}'=\left(
                   \begin{array}{cc}
                     0 & B_1 \\
                     -B_1^tI_{1,3} & 0 \\
                   \end{array}\right)\dd z.\end{equation}
 \end{definition}
\begin{remark}
 Note that for a
 harmonic map to be conformally harmonic, one only needs
\[\langle\alpha_{\mathfrak{p}}'(\frac{\partial}{\partial z}),\alpha_{\mathfrak{p}}'(\frac{\partial}{\partial z})\rangle =
tr \left( \left(\alpha_{\mathfrak{p}}'(\frac{\partial}{\partial z})\right)^tI_{1,3}\alpha_{\mathfrak{p}}'(\frac{\partial}{\partial z})\right)=0.\]
But this follows immediately from the condition on $B_1$ we have assumed. The same condition now shows that $f$ even is strongly conformally harmonic.
\end{remark}

Applying  to Willmore surfaces, we derive
in view of equation \eqref{B1} immediately
\begin{corollary} \label{nilpotency} The conformal Gauss map of a strong Willmore map is a strongly conformally harmonic map.
\end{corollary}

\begin{theorem} \label{lemma-B-12}
Let $U$ be a contractible open Riemann surface with local complex coordinate $z$. Let $f: U\rightarrow SO^+(1,n+3)/SO^+(1,3)\times SO(n)$ be a strongly conformally harmonic map
 with  frame
$F$, Maurer-Cartan form $\alpha$ and the (1,2) blocked part $B_1$ of $\alpha_{\mathfrak{p}}(\frac{\partial}{\partial z})$ as above. Then $B_1$ has, after some gauge or a gauge and a change of orientation on $U$ if necessary, the form
\begin{equation}\label{eq-B1}
B_1=\left(
      \begin{array}{ccc}
        \sqrt{2}\beta_1 & \cdots &\sqrt{ 2}\beta_n \\
        -\sqrt{2}\beta_1 & \cdots &  -\sqrt{ 2}\beta_n \\
        -k_1 & \cdots & -k_n \\
        -ik_1 & \cdots & - ik_n \\
      \end{array}
    \right).\end{equation}

    \end{theorem}

\begin{proof}
The conformal harmonicity of $f$ ensures that $B_1$ is a real analytic matrix function (\cite{Ee-S}, \cite{Ku-Sch}).
By Theorem \ref{normalizationlemma}, there exists $A:U \rightarrow SO^+(1,3)$ such that
{\small \begin{equation*}
AB_1=\left(
      \begin{array}{ccc}
        \sqrt{2}\beta_1 & \cdots &\sqrt{ 2}\beta_n \\
        -\sqrt{2}\beta_1 & \cdots &  -\sqrt{ 2}\beta_n \\
        -k_1 & \cdots & -k_n \\
        -ik_1 & \cdots & -ik_n \\
      \end{array}
    \right),\ \hbox{or }
    AB_1=\left(
      \begin{array}{ccc}
        \sqrt{2}\beta_1 & \cdots &\sqrt{ 2}\beta_n \\
        -\sqrt{2}\beta_1 & \cdots &  -\sqrt{ 2}\beta_n \\
        -k_1 & \cdots & -k_n \\
        ik_1 & \cdots & ik_n \\
      \end{array}
    \right)\ \hbox{ on } U.
\end{equation*}}

For the first case, setting
$\hat F= F\cdot\hbox{diag}(A,I_n)
$ we obtain $\hat{B}_1= AB_1$  on $U$.
For the second case, setting $w=\bar{z}$ induces an opposite orientation on $U$ and $U$ is also a Riemann surface for this new coordinate. Now $A\bar{B}_1$ is of the desired form.\end{proof}

Identify $f(p)$ with $V_p$, $ p\in M$, where $V_p$ is the oriented Lorenzian 4-subspace of $\mathbb{R}^{n+4}_1$.
Similar to Theorem \ref{minimal}, we will say that {\em $f$ contains a constant light-like vector $Y_0$ }if there exists a non-zero constant lightlike vector $Y_0$ in $\mathbb{R}^{n+4}_1$ satisfying $Y_0\in V_{p}$ for all $p\in M$.
\begin{theorem}\label{th-Willmore-harmonic-U}
Let $U$ be a contractible open Riemann surface with local complex coordinate $z$. Let $f: U\rightarrow SO^+(1,n+3)/SO^+(1,3)\times SO(n)$ be a harmonic map with  frame
$F=(e_0,\hat{e}_0,e_1,e_2,\psi_1,\cdots,\psi_n):U \rightarrow SO^+(1,n+3)$, Maurer-Cartan form $\alpha$ and the  $(1,2)-$block $B_1$ of $\alpha_{\mathfrak{p}}(\frac{\partial}{\partial z})$ occurring  in $($\ref{eq-B0}$)$. Assume moreover that $f$ is a strongly conformally harmonic map on $U$, that is,  $B_1^tI_{1,3} B_1 =0$.
Then w.l.g., $B_1$ has the form \eqref{eq-B1} on $U$, after a change of orientation of $U$ if necessary.
Writing the $(1,1)-$block $A_1$ of $\alpha_{\mathfrak{k}}(\frac{\partial}{\partial z})$ occurring  in $($\ref{eq-B0}$)$  in the form
\begin{equation}
A_1=\left(
         \begin{array}{ccccc}
         0 &  a_{12} &  a_{13} & a_{14} \\
          a_{12}  & 0 &  a_{23} & a_{24} \\
            a_{13} & -a_{23} & 0 & a_{34} \\
            a_{14}  & -a_{24}& -a_{34}  & 0 \\
         \end{array} \right),
\end{equation}
we distinguish two cases:

 \begin{enumerate}[$(a)$]
\item{ $a_{13}+a_{23}\not\equiv0$ on $U$:}
 In this case, there exists an open dense subset $U\backslash U_0$ such that $a_{13}+a_{23}\neq0$ on $U\backslash U_0$ and  $a_{13}+a_{23}=0$ on $ U_0$.  And $f$ is the conformal Gauss map of the unique Willmore surface $y=[e_0-\hat {e}_0]:U\setminus U_0\rightarrow S^{n+2}$. Moreover, $y$ has a conformal extension to $U$,  but is not an immersion on $U_0$.
  \begin{enumerate}[$(1)$]
  \item If the maximal rank of $B_1$ is $2$,  $y$ is not S-Willmore.
\item  If
the maximal rank of $B_1$ is $1$, then $y$ is S-Willmore on $U\setminus U_0$. Moreover, its dual surface $\hat y$ has a conformal extension to $U$, and after changing the orientation of $U$, $f$ will be the conformal Gauss map of the dual surface $\hat y$ on the points where $\hat y$ is immersed.\vspace{2mm}
 \end{enumerate}

\item {    $a_{13}+a_{23}\equiv 0$ on $U$:}
  In this case,  $f$ contains a constant light-like vector.
  \begin{enumerate}[$(1)$]
  \item   If the maximal rank of $B_1$ is $2$, then $f$ can not (even locally)  be the conformal Gauss map of any Willmore map.\vspace{1mm}
\item    If the maximal rank of $B_1$ is $1$, $f$ belongs to one of the following two cases:
  \begin{enumerate}[$(i)$]
  \item There exists some open and dense subset $U^*$ of $U$ such that $f$ is the conformal Gauss map of some uniquely determined Willmore surface $y^*:  U^* \rightarrow S^{n+2}$, either for $U$ with the given complex structure or the conjugate complex structure.
 Moreover, $y^*$ is conformally equivalent to a minimal surface in $R^{n+2}$, and $y^*$ can be extended smoothly to $U$. The conformal map
 $ y^*$ may be branched or un-branched on $U \setminus U^*$.
  \item $f$ reduces to a conformally harmonic map into $SO^+(1,n+1)/SO^+(1,1)\times SO(n)
$ or into $SO(n+2)/SO(2)\times SO(n)$, considered as natural submanifolds of $SO^+(1,n+3)/SO^+(1,3)\times SO(n)$.
 In this case $f$ is not (even locally) the conformal Gauss map of a Willmore map.
 \end{enumerate} \end{enumerate}
\end{enumerate}
\end{theorem}

\begin{proof}

As remarked above, the condition on $B_1$ implies that $f$ is a conformally harmonic, even strongly conformally harmonic map.
Moreover, by Theorem  \ref{normalizationlemma} and Theorem  \ref{lemma-B-12}  we can assume w.l.g. (after changing the complex structure of $U$, if necessary) that $B_1$ has the form of (\ref{eq-B1}).
  The proof of parts (a) and (b) is based on an evaluation of the third of the harmonic map equations (\ref{harmonic}).
Writing this equation in terms of matrix entries we obtain:
 \begin{subequations} \label{harmonic-3}
\begin{align}
& ~~ \beta_{j\bar{z}}-\bar{a}_{12}\beta_j-\frac{\sqrt{2}}{2}(\bar{a}_{13}+i\bar{a}_{14})k_j-\sum_{l=1}^n\beta_l\bar{b}_{lj}=0       \label{harmonic-3:1A} \\
& -\beta_{j\bar{z}}+\bar{a}_{12}\beta_j-\frac{\sqrt{2}}{2}(\bar{a}_{23}+i\bar{a}_{24})k_j+\sum_{l=1}^n\beta_l\bar{b}_{lj}=0,    \label{harmonic-3:1B} \\
 &-k_{j\bar{z}}+\sqrt{2}(\bar{a}_{13}+\bar{a}_{23})\beta_j-i\bar{a}_{34}k_j+\sum_{l=1}^n k_l\bar{b}_{lj}=0, \label{harmonic-3:1C}\\
& -ik_{j\bar{z}}+\sqrt{2}(\bar{a}_{14}+\bar{a}_{24})\beta_j+\bar{a}_{34}k_j+i\sum_{l=1}^n k_l\bar{b}_{lj}=0, \ \ j=1,\cdots,n.\label{harmonic-3:1D}
\end{align}
\end{subequations}
By \eqref{harmonic-3:1A}$+$\eqref{harmonic-3:1B} and \eqref{harmonic-3:1C}$+i\cdot$\eqref{harmonic-3:1B} one obtains
 \[(a_{13}+a_{23}-i(a_{14}+a_{24})) \bar{k}_j=(a_{13}+a_{23}-i(a_{14}+a_{24}))\bar{\beta}_j=0, \ j=1,\cdots,n.\]
Since $f$ is non-constant, not all the $\beta _j$ and all the $k_j$ vanish and we infer
\begin{equation}\label{spec-cond}
a_{13}+a_{23}=i(a_{14}+a_{24}).
\end{equation}

Recall that $F=(e_0,\hat{e}_0,e_1,e_2,\psi_1,\cdots,\psi_n)$. Set $Y_0=\frac{1}{\sqrt{2}}(e_{0}-\hat{e}_{0})$ and $ N_0=\frac{1}{\sqrt{2}}(e_{0}+\hat{e}_{0})$. We have
\begin{equation}\label{eq-ez}
  \left\{\begin{split}
  e_{0z}&=a_{12}\hat e_0+a_{13}e_1+a_{14}e_2+\sqrt{2}\sum_{1\leq j\leq n}\beta_j\psi_j,\
   \hat e_{0z} =a_{12} e_0-a_{23}e_1-a_{24}e_2+\sqrt{2}\sum_{1\leq j\leq n}\beta_j\psi_j,\\
  e_{1z}&=a_{13}e_0+a_{23}\hat e_0-a_{34}e_2+\sum_{1\leq j\leq n}k_j\psi_j, \
  e_{2z} =a_{14}e_0+a_{24}\hat e_0+a_{34}e_1+\sum_{1\leq j\leq n}ik_j\psi_j.\\
  \end{split}\right.
\end{equation}
Then
\[Y_{0z}=\frac{1}{\sqrt{2}}(e_{0}-\hat{e}_{0})_z=-a_{12}Y_{0}+\frac{1}{\sqrt{2}}(a_{13}+a_{23})(e_1-ie_2)\]
follows. Now there are two possibilities:\vspace{2mm}

{\bf Case (a):} $a_{13}+a_{23}$ does not  vanish identically on $U$. Since  $a_{13}+a_{23}$ is real analytic,
there exists some subset $U_0$  of $U$ satisfying the first part of the claim. In this case one verifies directly that $[Y_{0}]$ is conformal from $U$ to $S^{n+2}$ and a conformal immersion from $U \setminus U_0$  into $S^{n+2}$.

Calculating
\[Y_{0z\bar{z}}=|a_{13}+a_{23}|^2N_0 \mod \{Y_{0},e_1,e_2 \}\]
shows that $f$ is the harmonic conformal Gauss map of  $[Y_{0}]$ on $U\setminus U_0$. As a consequence,  $[Y_0]$ is a Willmore surface on $U \setminus U_0$.

 The claims involving the   S-Willmore  condition follow from Corollary \ref{S-frame}. So we only need  to consider Case (a2). By Theorem \ref{thm-dual gauss map}, the dual surface $\hat y$ of $[Y_0]$ is defined on $U\setminus U_0$ and has $f$ as its conformal Gauss map if we change the complex structure of $U$,  if $\hat y$ is immersed on an open dense subset of $U\setminus U_0$.

It remains to prove in this case that $\hat{y}$ extends real analytically to all of $U$. By Theorem \ref{thm-dual gauss map}, $\hat y$ is defined on $U\setminus U_0$. To extend $\hat y$ to $U$,  consider a lift $\hat Y$ of $\hat y$. Then  $\hat{Y} = a_0 Y_0 + \check{a}_0 N_0 + a_1 e_1 + a_2 e_2$ for some real valued functions defined on $U\setminus U_0$.
Since $\hat{Y}$ is not always a multiple of $Y_0$, it follows that $Y_0$ and $\hat{Y}$ are linearly independent on some open (and dense) subset $U'(\subset U\setminus U_0)$ of $U$, and $\check{a}_0 \neq 0$  on $U'$ follows. We can thus consider $\frac{1}{\check{a}_0}\hat Y$ on $U'$. Hence we can assume w.l.g. that
$\hat Y$ is of the form
\begin{equation}\label{hat-Y}
\hat Y =N_{0}+\mu_{1}e_{1}-\mu_{2}e_{2}+\frac{1}{2}|\mu|^2 Y_{0}
\end{equation}
 on $U'$, with $\mu=\mu_1+i\mu_2$ a real analytic complex valued function.
Using \eqref{eq-ez} we obtain
\begin{equation}
\hat Y_{z}=\sum_{j}^{n}(2\beta_j+\bar\mu k_j)\psi_j \mod \{e_{0},\hat{e}_{0},e_1,e_2\} \hbox{ on }  U'.
\end{equation}
Now the duality condition implies  $\hat Y_{z} = 0 \mod \{e_{0},\hat{e}_{0},e_1,e_2\},$ whence $\mu$ is the solution to the equations $\beta_j=-\frac{\bar\mu}{2}k_j,\ j=1,2,\cdots,n$ on $U'$.

To extend $\hat Y$ all we need to do is to extend the definition of $\mu$ real analytically to $U$. To extend $\mu$, we only need to show that $\mu$ has a well-defined (finite or infinite) limit at all points $U^*\subset U$ where $\sum|k_j|^2$ vanishes. From Corollary \ref{corollary-B1-vanish} in Appendix B we know $B_1=h_0\tilde B_1$ with $\tilde B_1$ never vanishing on $U$. So for every $p\in U^*$, there exists some $j$ such that $\beta_j=h_0\tilde \beta_j,\ k_j=h_0\tilde k_j$  with $|\tilde\beta_j(p)|^2+|\tilde{k}_j(p)|^2\neq0$.
 If    $\tilde{k}_j(p) \neq 0$, then the limit $\lim_{\tilde p\rightarrow p}\frac{\beta_j}{k_j}=\frac{\tilde{\beta_j}}{\tilde{k}_j}$ exists, is finite and the quotient function is real analytic in a neighbourhood of $p$.
In this case we can define $[\hat Y ]$ also by \eqref{hat-Y}.
If $\tilde{k}_j=0$, then $\tilde{\beta}_j \neq 0$ and the inverse of the limit is real analytic in a neighbourhood of $p$.
Then we consider  (locally) $\tilde{\hat {Y}}=\frac{1}{|\mu|^2}\hat Y$. By
the argument just given $\tilde{\hat{Y}}$ is well-defined and real analytic in a (possibly small) neighbourhood $U_p$ of $p\in U$ and $[\tilde{\hat{Y}}]=[\hat Y]=\hat{y}$ holds on $U_p\cap U'$.
  \vspace{3mm}

{\bf Case (b):} $a_{13}+a_{23}=0$ on $U$. Then we obtain
$Y_{0z}=-a_{12}Y_{0}.$ By scaling, we may assume that $Y_{0z}=0$ holds. Hence we can assume w.l.g. $a_{12} = 0$ on $U$.\vspace{2mm}

 (b1): Let's assume that the strongly conformally harmonic map $f$ is the conformal Gauss map of some conformal
immersion $\check{y}: U' \rightarrow S^{n+2},$ where $U'$ is some (possibly small) open subset of $U$. In this case, the canonical lift $\check{Y}$ of $\check{y}$ satisfies
(with $z = u + iv$)
\begin{equation} \label{checkY=f}
 f = \hbox{Span}_{\mathbb{R}} \{ \check{Y},\check{Y}_u, \check{Y}_v, \check{Y}_{z \bar{z}}\} = \hbox{Span}_{\mathbb{R}}\{e_{0},\hat{e}_{0},e_1,e_2\} = \hbox{Span}_{\mathbb{R}} \{ Y_0, N_0, e_1, e_2 \}.
\end{equation}
Thus similar to the discussions above, we can thus assume that a lift $Y_{\mu}$ of $\check y$ is of the form
\begin{equation} \label{Y-mu}
Y_{\mu}=N_{0}+\mu_{1}e_{1}-\mu_{2}e_{2}+\frac{1}{2}|\mu|^2 Y_{0}
\end{equation}
with $\mu=\mu_1+i\mu_2$ a real analytic complex valued function defined on an open and dense subset $U''$  of $U$.
Then,  since  $[Y_{\mu}]=\check{y}$ is a conformal surface with  conformal Gauss map $f$, we know that $Y_{\mu}$ satisfies
\begin{equation} \label{b2}
\ Y_{\mu z} \in
\hbox{Span}_{\mathbb{C}}\{e_{0},\hat{e}_{0},e_1,e_2\}.
\end{equation}
Using \eqref{eq-ez} we obtain
\begin{equation} \label{Ymuz1}
 Y_{\mu z}=\sum_{j}^{n}(2\beta_j+\bar\mu k_j)\psi_j \mod \{e_{0},\hat{e}_{0},e_1,e_2\} \hbox{ on }  U''.
\end{equation}
Now the last two equations imply  $Y_{\mu z} = 0 \mod \{e_{0},\hat{e}_{0},e_1,e_2\},$ whence
$\beta_j=-\frac{\bar\mu}{2}k_j,\ j=1,2,\cdots,n$ on $U''$.
From this we infer that on $U''$ we have rank $B_1 \leq 1$ and
since $B_1$ is real analytic we obtain $rank(B_1)\leq 1$ on $U$ and the claim follows.\vspace{2mm}

(b2): Let's assume now that the maximal rank of $B_1$ is $1$. We distinguish two cases according to the vanishing or not of $\sum |k_j|^2$ on $U$.\vspace{2mm}

Case (b2.a): $\sum |k_j|^2 \neq 0$ and rank$(B_1) = 1$ on the open  and dense subset $U^\sharp$ of $U$:  Note that these conditions imply
$\beta_j=-\frac{\bar\mu}{2}k_j, j = 1, \dots n$ on $U^\sharp$ for some function $\mu$ on the open (and dense) subset $U^\sharp$ of $U$.
Now we consider $Y_\mu$ of the form \eqref{Y-mu} with this $\mu$. We will show that it satisfies $Y_{\mu},\ Y_{\mu z},\ Y_{\mu z\bar{z}}\in
\hbox{Span}_{\mathbb{C}}\{e_{0},\hat{e}_{0},e_1,e_2\}.$
It will turn out to be convenient to rewrite $Y_\mu$ in the form
\begin{equation} \label{Y-mu-2}
Y_{\mu}=N_{0}+\mu_{1}e_{1}-\mu_{2}e_{2}+\frac{|\mu|^2}{2}Y_{0}=
N_0+\bar\mu P+\mu \bar{P}+\frac{|\mu|^2}{2}Y_{0},
\end{equation}
with $\mu=\mu_1+i\mu_2$ and  $\ P=\frac{1}{2}(e_1-ie_2)$. Note that for $Y_\mu$ equation (\ref{Ymuz1}) holds. Substituting $\beta_j=-\frac{\bar\mu}{2}k_j$ into the expression (\ref{Ymuz1}) for
$Y_{\mu z}$ we obtain
$Y_{\mu u}, Y_{\mu v}  \in \hbox{Span}_{\mathbb{R}}
\{ e_{0}, \hat{e}_{0},e_1,e_2 \}$.

Moreover,  using $P_{\bar{z}}=-i\bar{a}_{34}P
+\frac{1}{2}(\bar{a}_{13}-i\bar{a}_{14})Y_0,$
which follows from $\dd F = F \alpha$ together with the special values for the entries of $\alpha$  in the case under discussion, one derives
$ Y_{\mu z\bar{z}}\in \hbox{Span}_\R \{e_{0},\hat{e}_{0},e_1,e_2\}.$
Thus we have shown that for the lightlike vector $Y_\mu$ the relation
$Y_{\mu},\ Y_{\mu z},\ Y_{\mu z\bar{z}}\in
\hbox{Span}_{\mathbb{C}}\{e_{0},\hat{e}_{0},e_1,e_2\} $ holds.

 Next we want to determine, when the map $Y_\mu$
  comes from some conformal map into $S^{n+2}$. It is easy to verify
that  $y^* = [Y_{\mu}]$ is a conformal (whence Willmore) surface with its conformal Gauss map spanning the same vector space as $f$ if and only if $Y_{\mu}$ satisfies
\begin{equation} \label{b2}
Y_{\mu},\ Y_{\mu z},\ Y_{\mu z\bar{z}}\in
\hbox{Span}_{\mathbb{C}}\{e_{0},\hat{e}_{0},e_1,e_2\},
\ \langle Y_{\mu z},Y_{\mu z}\rangle=0,\ \langle Y_{\mu z},Y_{\mu \bar{z}}\rangle>0.
\end{equation}
We have shown above that the first three conditions are satisfied for $Y_\mu$. To verify the fourth condition we evaluate the harmonicity condition for $f$.
Substituting $\beta_j=-\frac{\bar\mu}{2}k_j$
into the  first equation of  \eqref{harmonic-3} and using  $a_{13}+a_{23}=0$, $a_{12} = 0$, and the third equation of \eqref{harmonic-3} we derive
$\mu_z+\sqrt{2}(a_{13}-ia_{14})+ia_{34}\mu=0$ on $U^\sharp$.

Next we observe that the derivative of  $Y_\mu$ also satisfies
\begin{equation}
Y_{\mu z}=(\cdots)P+(\cdots)Y_0 +(\mu_z+\sqrt{2}(a_{13}-ia_{14})+ia_{34}\mu)\bar{P} = (\cdots)P+(\cdots)Y_0,
\end{equation}
whence $\langle Y_{\mu z},Y_{\mu z}\rangle=0$ follows.

As a consequence, the only condition to decide whether $f$  corresponds to a conformal  immersion or not is, whether we can satisfy
 $ \langle Y_{\mu z},Y_{\mu \bar{z}}\rangle>0$, or not.

This naturally leads to the two cases listed in the theorem.

 Case (i). There exists an  open and dense subset  $U^* \subset U^\sharp \subset U$,  where $ \langle Y_{\mu z},Y_{\mu \bar{z}}\rangle>0$ holds.
Then, for every $z\in U^*$, the subspace spanned by $f(z,\bar z)$  coincides with the one of $Gr_{y^*}(z,\bar z)$. Hence either $f(z,\bar z)=Gr_{ y^*}(z,\bar z)$ on $U^*$, or $f(z,\bar z)$ spans the same subspace as
$Gr_{y^*}(z,\bar z)$ and has an orientation which is opposite to the one of $Gr_{ y^*}(z,\bar z)$. In the latter case this just says that after a change of complex structure of $U$ the conformal Gauss map of $y^*$ coincides with $f$.
Moreover, since $f$ and hence $Gr_{y^*}$ contain the  light-like vector $Y_0$,
by the stereographic projection with respect to $Y_0$, $[Y_{\mu}]$ becomes a minimal
surface in $R^{n+2}$ (\cite{Bryant1984}, \cite{Ejiri1988},\ \cite{Ma-W1}).

To extend ${y^*}=[Y_{\mu}]$ to  $U$, we need only to show that $\mu$ has a  well-defined ( finite or infinite)  limit at all points where $\sum|k_j|^2$ vanishes. The argument for this is very similar to the argument given  in Case $(a2)$.

From Corollary \ref{corollary-B1-vanish} in Appendix B we know   $B_1=h_0\tilde B_1$ with $\tilde B_1$ never vanishing on $U$. So for every $p\in U\backslash U^{\sharp}$, there exists some $j$ such that $\beta_j=h_0\tilde \beta_j,\ k_j=h_0\tilde k_j$  with $|\tilde\beta_j(p)|^2+|\tilde{k}_j(p)|^2\neq0$.
 If    $\tilde{k}_j(p) \neq 0$, then the limit $\lim_{\tilde p\rightarrow p}\frac{\beta_j}{k_j}=\frac{\tilde{\beta_j}}{\tilde{k}_j}$ exists, is finite and the quotient function is real analytic in a neighbourhood of $p$.
In this case we can define $[Y_{\mu}]$ also by \eqref{Y-mu-2}.
If $\tilde{k}_j=0$, then $\tilde{\beta}_j \neq 0$ and the inverse of the limit is real analytic in a neighbourhood of $p$.
Then we consider  (locally) $\tilde{Y}_{\mu}=\frac{1}{|\mu|^2}Y_{\mu}$. By
 the argument just given  $\tilde{Y}_{\mu}$ is well-defined and real analytic in a (possibly small) neighbourhood $U_p$ of $p\in U$ and $[\tilde{Y}_{\mu}]=[Y_{\mu}]={y^*}$ holds on $U_p\cap U^{\sharp}$. \vspace{1mm}

  Case (ii).
If  $ \langle Y_{\mu z},Y_{\mu \bar{z}}\rangle=0$ on $U$, then
$Y_{\mu}$ is another constant light-like vector of $f$, linearly independent from $Y_0$. So $Y_0$ and $Y_\mu$ span a constant, real, 2-dimensional Lorentzian subspace. Let $\{\tilde e_0,\tilde{\hat{e}}_0\}$ be an Euclidean oriented orthonormal basis of it and let $\{\tilde e_0,\tilde{\hat{e}}_0, \tilde e_1, \tilde e_2\}$ be an oriented orthonormal basis of $Span_{\R}\{e_0,\hat{e}_0,e_1,e_2\}$. Since $\{e_0,\hat{e}_0\}$ and
$\{\tilde{e}_0,\tilde{\hat{e}}_0\}$ span 2-dimensional Lorentzian subspaces and $\{e_1,e_2\}$ and $\tilde{e}_1, \tilde{e_2}$ span Euclidean subspaces, there exists a transformation in $SO^+(1,3)$ which maps the first basis onto the second one in the given order.
Then the lift $\tilde{F}(z,\bar z)=( \tilde e_0,\tilde{\hat{e}}_0, \tilde e_1, \tilde e_2,\psi_1,\cdots, \psi_n)$ reduces to a map into $SO(n+2)\subset SO^+(1,n+3)$, i.e, $f$ reduces to a harmonic map into $SO(n+2)/SO(2)\times SO(n)
 \subset SO^+(1,n+3)/SO^+(1,3)\times SO(n)$. \vspace{2mm}

Case $(b2.b)$: $\sum |k_j|^2 \equiv 0$: In this case the integrability condition of $\alpha$ implies, in view of the vanishing of both sides on (\ref{b2}), that the submatrix of $\alpha$ with entries $33,34,43,44$ is integrable. Hence, after gauging $\alpha$ by some matrix in $SO(2)$ we can assume w.l.g. that $a_{34} = 0$ holds.
Hence we obtain
 $e_{1z}=\sqrt{2}a_{13}Y_0$ and $\ e_{2z}=\sqrt{2}a_{14}Y_0.$  Since $Y_{0z}=0$, similarly after gauging $\alpha$ by some matrix in $SO^+(1,3)$ we can assume w.l.g. that $a_{13}=a_{14} = 0$ holds, i.e., $e_1$ and $e_2$ are constant.
Hence
$f$ reduces to a map into  $SO^+(1,n+1)/SO^+(1,1)\times SO(n)
 \subset SO^+(1,n+3)/SO^+(1,3)\times SO(n)$.
\end{proof}

\begin{remark} Note that the proof of Theorem \ref{th-Willmore-harmonic-U} shows how one can construct a Willmore surface from a strongly conformally harmonic map $f:U \rightarrow SO^+(1, n+3)/{SO^+(1,3) \times SO(n)}$  or prove that $f$ is not the conformal Gauss map of any conformal map $y: U \rightarrow S^{n+2}$:
\begin{enumerate}
  \item
 Step 1: Choose any frame $\ \tilde{F}: U \rightarrow SO^+(1, n+3)$.

  \item Step 2: Choose a gauge  $A: U \rightarrow SO^+(1,3) \times SO(n)$ such that the Maurer-Cartan form  $ \alpha = F^{-1} d F$
of the new frame  $F= \tilde{F} A $ has the form as stated in (\ref{eq-B1}).
 For this we may need to change the complex structure on $U$.
 We write  $F=(e_0,\hat{e}_0,e_1,e_2,\psi_1,\cdots,\psi_n)$as in the theorem
 and form $Y_0=\frac{1}{\sqrt{2}}(e_{0}-\hat{e}_{0}).$
 Then in the proof above one shows
 \begin{equation} \label{cases}
 Y_{0z}=\frac{1}{\sqrt{2}}(e_{0}-\hat{e}_{0})_z=-a_{12}Y_{0}+\frac{1}{\sqrt{2}}(a_{13}+a_{23})(e_1-ie_2).
 \end{equation}
So $ [Y_{0}]$ is constant if and only if  $a_{13} + a_{23} \equiv 0$  if and only if the frame $F$ is in {\bf Case (b)}.

  \item Step 3: Consider the function $h = a_{13} + a_{23} : U \rightarrow \C$.
\begin{enumerate}
  \item
 Step 3a: $h \not\equiv 0$ on $U$:

 Then $f$ is the (harmonic) oriented conformal Gauss map of the conformal map $y=[Y_0]=[e_0 - \hat{e}_0] :\hat{U} \rightarrow S^{n+2}$.
It turns out that this is the generic case, see below.

 \item  Step 3b: $h  \equiv 0$ on $U$:

 If the maximal rank of $B_1$ is 2, then $f$ is not (even locally) the conformal Gauss map of any conformal immersion.

  If the maximal rank of $B_1$ is 1,
Consider the function
$p = \sum_{j=1}^n  | k |^2 : U \rightarrow \R _{\geq 0}$.  If $p \equiv 0$ on $U$, then $f$ is not (even locally) the conformal Gauss map of any conformal immersion. If $p \not\equiv 0$ on $U$, we can only obtain (possibly only after changing the complex structure on U) Willmore surfaces in $S^{n+2}$ which are conformal to minimal surfaces in $\R^{n+2}$. If we are not interested in minimal surfaces in $\R^{n+2}, $ then we are done. Otherwise let $\mu =-\frac{2\bar{\beta}_j}{\bar{k}_j}$ on the points $\bar{k}_j\neq 0$. With such a function $\mu$ we consider $Y_\mu$ as in (\ref{Y-mu}).   The stereographic projection of  $Y_\mu$  with center $Y_0 $ yields a minimal surface in $\R^{n+2}$.
  \end{enumerate}
\end{enumerate}
\end{remark}

\begin{remark} \
\begin{enumerate}
\item Note that if $rank B_1=2$, there is a unique gauge of the frame such that Theorem 3.4 holds. Correspondingly, there is a unique map $[Y_0]$ associated to $f$.
If $rank B_1=1$, there are two different kinds of gauges of the frame $F$ such that Theorem 3.4 holds. Consequently there are two different projections from $F$ to $S^{n+2}$, giving a pair of Willmore surfaces dual to each other in general. This re-interprets the duality theorem of Willmore surfaces due to Blaschke \cite{Blaschke}, Bryant \cite{Bryant1984} and Ejiri \cite{Ejiri1988}.
    \item  When we begin with the harmonic map $f$, then one of the associated Willmore maps
     $[Y_0]=[e_{0}-\hat{e}_{0}]$ may degenerate to a point. If $rank B_1=2$, then we will obtain no Willmore surface associated to $f$.
        If $rank B_1=1$, there is another Willmore map $[\hat Y_0]$ associated to $f$. If $[\hat Y_0]$ does not degenerate to a point,  it is conformally equivalent to a minimal surface in
        $\R^{n+2}$. This is exactly how minimal surfaces in $\R^{n+2} $ occur in the  classification of Willmore 2-spheres in $S^3$ by Bryant \cite{Bryant1984} and the classification of  Willmore 2-spheres with dual surfaces in $S^{n+2}$ by Ejiri \cite{Ejiri1988}. We refer to the Corollary below for a description of minimal surfaces in $\R^{n+2}$ in terms of $f$.
\item
It turns out that the associated Willmore map/maps being non-degenerate is the generic case and of most interest in the loop group approach. Moreover if one uses the loop group formalism to produce all strongly conformally harmonic maps $f$ one can recognize immediately by looking at the ``normalized potential" if the associated harmonic map has a constant lightlike vector.  It is one of the main results of \cite{Wang} to show how these exceptional normalized potentials looks like. Excluding this exceptional case, all other normalized potentials yield harmonic maps with frames belonging to the case the associated Willmore map/maps being non-degenerate.
Since minimal surfaces in $\R^{n+2}$ can be well investigated in a simpler way,
we will be primarily  interested in the conformally harmonic maps with non-degenerate associated Willmore map/maps.
\item It is in general very hard to detect whether $y$ is immersed or branched at some point from the behaviour of $f$. It will be an interesting question to discuss the immersion property of $y$ in terms of $f$.
\end{enumerate}\end{remark}

Considering the non degenerate case, we have

\begin{theorem} Let $f$ be a strongly conformally harmonic map as in Theorem \ref{th-Willmore-harmonic-U}. Assume that $f$ does not contain any
constant lightlike vector ( so $f$ belongs to Case $(a)$).

\begin{enumerate}
  \item If  $rank B_1 = 2$, then there exists a  unique  Willmore map  $y:U\rightarrow S^{n+2}$ which is not an
S-Willmore  map,
such that $y$ is immersed on $U \setminus U_0 $ and  has $f$ as  its conformal Gauss map.
\item If $rank B_1 = 1$, then there exists a pair of  S--Willmore maps $y,\hat{y}:U\rightarrow S^{n+2}$ which are dual to each other and  such that
\begin{enumerate}[$(i)$]
  \item on an open dense subset of $U$, $y$ is immersed and $f$ is its conformal Gauss map;
  \item on an open dense subset  of $U$, $\hat y$ is immersed and $f$ is its conformal Gauss map after a change of the orientation of $U$.
  \end{enumerate}
\end{enumerate}\end{theorem}
Ejiri's Willmore torus in $S^5$ (\cite{Ejiri1982}) provides an example for Case (1), and Veronese spheres in $S^{2m}$ (\cite{Mon}) provide examples for Case (2).

We  have characterizations of minimal surfaces in $\R^{n+2}$ (the degenerate case) as follows:
\begin{corollary} Let $f$ be a  strongly conformally harmonic map as in Theorem \ref{th-Willmore-harmonic-U}.
\begin{enumerate}
  \item
 If $f$ belongs to  Case (a)
as well as to  Case (b). Then, possibly after  changing  the complex structure  of $U$, $f$ is the conformal Gauss map of some minimal surface in $R^{n+2}$ (after putting $R^{n+2}$ conformally into $S^{n+2}$), and vice versa.
\item If $f$ contains a
constant lightlike vector, then either  $f$ does not correspond to any Willmore map, or $f$ corresponds to a Willmore map which is conformally equivalent to a minimal surface into $\R^{n+2}$ (possibly after  changing  the orientation of $U$).
\end{enumerate}
\end{corollary}

By Theorem
\ref{thm-dual gauss map} and Theorem \ref{th-Willmore-harmonic-U}, we  have the following
\begin{lemma}
Let $f$ be a strongly conformally harmonic map as in Theorem \ref{th-Willmore-harmonic-U}. Fix the orientation of the contractible open set $U$. Then
\begin{enumerate}
\item either  $f$   is  not the conformal Gauss map of any conformal map on any open subset of $U$;
\item or there exists a unique  Willmore map $y:U\rightarrow S^{n+2}$ such that $f$ is the oriented conformal Gauss map of $y$ on an open dense subset of $U$.
\end{enumerate}
\end{lemma}

\begin{proof} Case (3) of Theorem \ref{thm-dual gauss map} gives the unique corresponding between $y$ and $f$ and Theorem \ref{th-Willmore-harmonic-U} tells either $y$ reduces to a point, or $y$ is an immersion on an open dense subset, which is exactly the lemma.\end{proof}

Since any two points of a surface $M$ are contained in a contractible open subset of $M$ the corollary yields straightforwardly the following
\begin{theorem}\label{th-Willmore-harmonic}  Let $f: M\rightarrow SO^+(1,n+3)/SO^+(1,3)\times SO(n)$ be a non-constant strongly conformally harmonic map from a connected Riemann surface $M$. If
on a contractible open subset $U\subset M$,  $f$ is the {oriented} conformal Gauss map of some Willmore immersion $y:U\rightarrow S^{n+2}$, then there exists a unique conformal (Willmore) map $y:M\rightarrow S^{n+2}$ such that $f$ is the oriented conformal Gauss map of $y$ on an open dense subset $M_1$ of $M$ and  $\tilde y |_{U}=y$.
\end{theorem}

%MMMMMMMMMMMMMMMMMMMMMMMMMMMMMMMMMMMMMMMMM

%%%%%%%%%%%%%%%%%%%%%%%%%%%%%%%%%%%%%%

\section{Appendix B: Proof of Theorem \ref{normalizationlemma} }

Let $U$ be a contractible open subset of some Riemann surface $M$. Then the Maurer-Cartan form of any strongly conformally harmonic map
$f: M \rightarrow G/K$ is  real analytic on $U$.
Moreover, the matrix $B_1$ in \eqref{eq-B0} satisfies $B_1^t I_{1,3} B_1 = 0$
 by Definition \ref{stronglyconfharm}.  In particular, the columns of $B_1$ are orthogonal  complex null vectors relative to the quadratic form
defined by $I_{1,3}$. Our goal is to find a simple canonical form of $B_1.$

The first case to consider is, where $B_1$ consists of one column.
It is easy to verify that every non-vanishing fixed complex null vector $b \in \C^4$ can be mapped by $SO^+(1,3)$ into the space
\[\mathcal{N} = \C (1,-1,0,0)^t\ \hbox{ or into }\ \mathcal{N}_{\pm}= \C (0,0,1, \pm i)^t\] according to whether the real part of $b$ is lightlike (possibly including $0$ ) or spacelike respectively.  For a real analytic complex valued null vector function $b$ it is not possible, in general, to map $b$ by some real analytic matrix function $A \in SO^+(1,3)$ into one of these spaces only.
But if $B_1 = b$ corresponds to a non-trivial strongly conformally harmonic map, then one can map $b$ into the sum $\mathcal{N} \oplus  \mathcal{N}_+$ or  into $\mathcal{N} \oplus  \mathcal{N}_-$.

We start by proving the desired result in the case, where $b$ never vanishes on $U$.

\begin{lemma}\label{lemma-null}
Let $U$ be a contractible open subset of $\C$ and $b: U \rightarrow \C^4_1\backslash\{0\}$
a real analytic null vector.
Then there exists a real analytic map $A:U \rightarrow SO^+(1,3)$ such that
the function $Ab$ is contained in $\mathcal{N} \oplus  \mathcal{N}_+$, i.e. $Ab$ has the form $(p, -p, q,  i  q)^t$.
\end{lemma}
\begin{proof}
The proof is particularly easy if one realizes $\C^4 \cong Mat(2, \C)$ with quadratic form
\[\langle X,X^\prime \rangle =
 X_{11} X^\prime_{22}  - X_{12} X^\prime_{21} + X_{22} X^\prime_{11}  - X_{21} X^\prime_{12} .\]
In this realization the non-vanishing complex null vectors are exactly all matrices of rank 1, i.e. all non-vanishing matrices of determinant $0$.
As real form we choose $\R^4_1 \cong Herm(2,\C)$, the space of $2 \times 2-$ complex hermitian matrices.

The group $SO(4,\C) \cong (SL(2, \C) \times SL(2,\C) ) / \{ \pm I \}$ acts on $Mat(2,\C)$ by $(g,h).X = gXh^{-1}$. Then $SO^+(1,3) \cong SL(2,\C) / \{\pm I\}$ acts by $g.X = gX \bar{g}^t$.

In the spirit of what was said before the statement of the lemma, we want to transform  any real analytic map $X$ defined in $U$ with values in $Mat(2,\C)$
 into the complex space $\C E_{11} \oplus \C E_{21}$ by the operation
$X \rightarrow gX \bar{g}^t$, where $g \in SL(2, \C)$ is defined on $U$ and real analytic.

Now it is an easy exercise to verify that for any $z_0 \in U$ there is a matrix function, $q_\delta$, defined on some open neighbourhood $U_\delta \subset U$ of $z_0$ such that $X \bar{g}_\delta^t$ has  on $U_\delta$  an identically vanishing second column, if $det(X) \equiv 0$ on $U$.  Of course, then also
$g_\delta X \bar{g}_\delta^t$ has identically vanishing second column.

Next we consider $h_{\alpha \beta} = q_\alpha q_\beta^{-1}$. These matrix functions are defined on $U_\alpha \cap U_\beta$  and form a cocycle  relative to the covering given by the $U_\delta$. Moreover, this cocycle consists of upper triangular matrices of determinant $1$. Therefore, since $U$ is contractible, this cocycle is a co-boundary. Therefore  there exist upper triangular matrices $h_\delta$  of determinant $1$  and  defined on $U_\delta$ satisfying $ q_\alpha q_\beta^{-1}= h_\alpha h_\beta^{-1}$.
As a consequence  $g = h_\alpha^{-1} q_\alpha$ is defined on $U$ and
 the  second column  of  $gX\bar{g}^t$ vanishes identically on $U$.
\end{proof}

Now it is fairly straightforward to prove Theorem 3.4.
If the maximal rank of $B_1$ is $1$, then the argument would be easy, if all columns of $B_1$  would be real analytic multiples of one of the columns, say, the first column of  $B_1$. The actual argument follows in a sense the same idea, but is a bit more sophisticated. If the maximal rank of $B_1$ is $2$, then in the complex vector space spanned by  two generically linearly independent columns of $B_1$ one constructs a real vector which then implies quite directly  what we want in view of the condition $B_1^t I_{1,3} B_1 = 0$.\\

{\em Proof of Theorem \ref{normalizationlemma}:}

First we mention some result which is true for all harmonic maps into a symmetric space, namely that any such harmonic map can be constructed by the loop group method from ``holomorphic potentials''. The proof is as in \cite{DPW} and is not related in any way to the specific properties of conformally harmonic maps which we investigate.

Therefore, let's consider the holomorphic potential of the harmonic map $f$ (for a discussion we refer to \cite{DPW, Do-Wa12}).
 Let
\begin{equation}\label{eq-potential-h}
\xi=(\lambda^{-1}\xi_{-1}+\sum_{j\geq0}\lambda^j\xi_j)\dd z,\  \hbox{ with }\ \xi_{-1}=\left(
    \begin{array}{cc}
      0 & R_1 \\
      -R_1^tI_{1,3} & 0 \\
    \end{array}
  \right),
\end{equation}
be the corresponding holomorphic potential on $U$. Then there exist some real analytic matrices $S_1 \in SO^+(1,3,\C)$ and $S_2\in SO(n,\C)$, such that $B_1=S_1R_1S_2$  holds.

Our claim is equivalent to that there exists some real analytic matrix function $A:U\rightarrow SO^+(1,3)$ such that $AB_1$ has the form desired.

It is easy to see that it suffices to prove this special form for $Q_1 = S_1 R_1$.
Let's write $Q_1$ as a matrix of column vectors, $Q_1 = (q_1,...,q_n).$ Since we assume w.l.g. $B_1\neq 0,$ also $Q_1 \neq 0.$  Hence one of the columns of $Q_1$ does not vanish. Let's assume w.l.g. that the first column $q_1$ of $Q_1$ does not vanish identically. Then the corresponding first column $r_1$ of $R_1$ does not vanish identically. Since $r_1$ is holomorphic, one can factor out some holomorphic (product) function $h_1$ such that $r_1 = h_1 \hat{r}_1$, where $\hat{r}_1$ is holomorphic and never vanishes on $U$.
As a consequence, $q_1 = h_1 \hat{q}_1$, where the globally defined and real analytic map $\hat{q}_1$ never vanishes.

 From Lemma \ref{lemma-null} we obtain now that there exists some real analytic matrix function $A:U\rightarrow SO^+(1,3)$ such that $A \hat{q}_1$ has the desired form
\[A \hat{q}_1=aE_{11}+bE_{21}.\]
Hence also $ A q_1 = h_1  A\hat{q}_1$ has the desired form.\vspace{2mm}

 Let's assume next that $B_1$ has maximal rank $1$. Then  we claim that each column of $Q_1$ is a multiple of $\hat{q}_1$ and
 this multiple is holomorphic on $U$.
As a consequence, $AB_1$ has the desired form.

To prove the claim above, note that the relation between $A S_1 r_1$ and $A S_1 r_j$ can already be found between $r_1$ and $r_j$. By the argument above we can write
$r_1 = h_1 \hat{r}_1$ and $r_j = h_j \hat{r}_j$ with $\hat{r_1}$ and $\hat{r}_j$
never vanishing on $U$. Let $U'$ denote the discrete subset of points in $U$, where none of the occurring, not identically vanishing functions/vector entries, vanish. On this set one can show that an entry of $\hat{r}_1$ does not vanish identically if and only if the corresponding entry of $\hat{r}_j$ does not vanish identically.
Now it is easy to verify that $\hat{r}_j$ is a holomorphic multiple of $\hat{r}_1$. Whence the statement above. \vspace{2mm}

Next let's assume that the maximal rank of $B_1$ is $2$.
In this case we apply the argument given above for $q_1$  to each column of $B_1$, i.e. we write $q_j = h_j \hat{q}_j$, where $\hat{q}_j$ never vanishes on $U$. Note, the case $q_j \equiv 0$ corresponds to $h_j \equiv 0$ and $\hat{q}_j = const \neq 0$.
 We will also assume w.l.g. that the second column of $B_1$ does not vanish identically. Hence $\hat{q}_1$ and $\hat{q}_2$  never vanish  on $U$ and are linearly independent on an open and dense subset $\tilde{U}$ of $U$.

 For the following argument we realize again $\C^4 $ by $Mat(2,\C)$. As before we can apply the theorem above to $\hat{q}_1$ and can assume w.l.g. that $\hat{q}_1$ is a $2 \times 2-$matrix for which the second column is $0$. We will use the notation $\hat{q}_1 = aE_{11} + bE_{21}$ and note that by assumption $|a|^2 + |b| ^2 $ never vanishes on $U$.

  If $ab\equiv0$, then $a\equiv 0$ or $b\equiv0$ on $U$.
The nilpotency condition  $L^t I_{1,3} L=0$ for
$Q_1 = S_1 R_1$ implies that the claim of the theorem holds, after one more (constant) gauging if necessary.

 If $ab\neq 0$, then after applying a constant
$SL(2,\C )-$matrix, if necessary, we can assume w.l.g. that $a \neq 0$ and $b \neq 0$ on the open and dense subset $\tilde{U}$ of $U$.

 In this case, using that $\hat{q}_1$ and $\hat{q}_2$ are perpendicular to each other and to themselves,  it is straightforward to see by a computation on $\tilde{U}$ that $\hat{q}_2$
is  either of the form
$  \hat{a} E_{11}+\hat{b} E_{21} $,  or of the form \[ \hat{a} (aE_{11}+bE_{21})+\hat{b}( aE_{12}+bE_{22}) \hbox{ with } \hat{b} \neq0,\] where the coefficient functions are real analytic on  $\tilde{U}$.
For the first case, we are done, since the coefficients clearly extend to functions defined on $U$.

 In the second case one can show by a simple computation that the complex vector space spanned by $\hat{q}_1$ and $\hat{q}_2$ contains the hermitian matrix
\[ w_1 = |a|^2 E_{11} + \bar{a} b E_{21} + a \bar{b} E_{12} |b|^2.\]
Clearly, this matrix is defined on all of $U$. Moreover, the $SL(2, \C )-$matrix
$g = c_0 (\bar{b} E_{11} - a E_{12} + \bar{a} E_{21} + b E_{22} )$, with $c_0 = 1/{\sqrt{|a|^2 + |b|^2}}$ is a real analytic function on $U$ which transforms
$w_1$ into the matrix  $w_2 = (|a|^2 + |b|^2) E_{11}.$ As a consequence, after this transformation the complex vector space spanned by $q_1$ and $q_2$ contains the constant matrix function $q_0 = E_{11}$.

 By the construction carried out so far, the vectors $q_0, q_1,...$ all are perpendicular to each other and to themselves. In particular, $\langle q_0,q_j\rangle = 0$ and  $\langle q_j,q_j\rangle =0$  for $j>0$ implies by a straightforward computation that each of the matrices $q_j, j>0,$ has a vanishing second column or a vanishing second row. But the relation $\langle q_1, q_k\rangle =0, k>1,$ implies that all $q_k$ have the same type as $q_1$. Hence all
$q_j, j \geq 1,$ are   contained in either $\mathcal{N} \oplus  \mathcal{N}_+$  or $\mathcal{N} \oplus  \mathcal{N}_-$.

   \hfill   $\Box$

\begin{corollary} \label{corollary-B1-vanish}
Let $h_0$ denote the greatest common divisor of the holomorphic
functions
$h_j, j = 1, \dots, n,$ defined in the proof above, then $B_1 = h_0 \hat{B}_1$ and  $\hat{B}_1(z) \neq 0$ for all $z \in U$.
\end{corollary}

{\footnotesize{\bf Acknowledgements}\ \ This work was started when the second named author visited the Department of Mathematics of Technische Universit\"{a}t  M\"{u}nchen, and  the Department of Mathematics of Tuebingen University. He would like to express his sincere gratitude for both the hospitality and financial support. The second named author is thankful to Professor Changping Wang and Xiang Ma for their suggestions and encouragement.
The second named author was partly supported by the Project 11571255 of NSFC. The second named author is thankful to the ERASMUS MUNDUS TANDEM Project for the financial supports to visit the TU M\"{u}nchen.}

{\footnotesize

\def\refname{Reference}

}
\vspace{2mm}
{\footnotesize
\begin{multicols}{2}   %第四段分两栏并在两栏列缝间插入分隔线
Josef F. Dorfmeister

Fakult\" at f\" ur Mathematik,

TU-M\" unchen, Boltzmann str. 3,

D-85747, Garching, Germany

{\em E-mail address}: dorfm@ma.tum.de\\

Peng Wang

College of Mathematics \& Informatics, FJKLMAA,

Fujian Normal University, Qishan Campus,

Fuzhou 350117, P. R. China

{\em E-mail address}: {pengwang@fjnu.edu.cn}

\end{multicols}}
\end{document}